\documentclass{amsart}
\usepackage{amsmath}
\usepackage{amsfonts}
\usepackage{amssymb}
\usepackage{graphicx}
\usepackage{latexsym}
\usepackage{amscd}
\usepackage[all, cmtip]{xy}
\usepackage[utf8]{inputenc}
\usepackage{amsthm}
\usepackage{xcolor}
\usepackage[a4paper,  margin=1in]{geometry}
\usepackage{epic}
\usepackage{pstricks}
\usepackage[all]{xy}
\usepackage{graphicx}
\usepackage{epic,eepic}
\usepackage{epsfig}
\usepackage{float}
\usepackage{extarrows}
\usepackage{tikz}
\usepackage{calc}

\providecommand{\U}[1]{\protect\rule{.1in}{.1in}}
\providecommand{\U}[1]{\protect\rule{.1in}{.1in}}
\providecommand{\U}[1]{\protect\rule{.1in}{.1in}}
\newtheorem{theorem}{Theorem}[section]

\newtheorem{corollary}[theorem]{Corollary}

\newtheorem{definition}[theorem]{Definition}
\newtheorem{example}[theorem]{Example}

\newtheorem{lemma}[theorem]{Lemma}

\newtheorem{question}[theorem]{Question}
\newtheorem{proposition}[theorem]{Proposition}
\newtheorem{remark}[theorem]{Remark}

\DeclareMathOperator{\lcm}{lcm}
\DeclareMathOperator{\gcd1}{gcd}
\usetikzlibrary{decorations.markings}

\tikzstyle{vertex}=[circle, draw, fill, inner sep=0pt, minimum size=6pt]

\begin{document}
\title[Factorization of ideals in Leavitt Path Algebras]{Variations of
primeness and Factorization of ideals in Leavitt Path Algebras}
\author[Aljohani]{Sarah Aljohani}
\address{Department of Mathematics and Statistics, Saint Louis University,
St. Louis, MO-63103, USA}
\email{sarah.aljohani@slu.edu}
\author[Radler]{Katherine Radler}
\address{Department of Mathematics and Statistics, Saint Louis University,
St. Louis, MO-63103, USA}
\email{katie.radler@slu.edu}
\author[Rangaswamy]{Kulumani M. Rangaswamy}
\address{Departament of Mathematics, University of Colorado at Colorado
Springs, Colorado-80918, USA}
\email{krangasw@uccs.edu}
\author[Srivastava]{Ashish K. Srivastava}
\address{Department of Mathematics and Statistics, Saint Louis University,
St. Louis, MO-63103, USA}
\email{ashish.srivastava@slu.edu}
\thanks{The work of the fourth author is partially supported by a grant from
Simons Foundation (grant number 426367).}
\dedicatory{Dedicated to the loving memory of Anupam, the 7 year old budding
Geologist and Astronomer}

\begin{abstract}
In this paper we describe three different variations of prime ideals:
strongly irreducible ideals, strongly prime ideals and insulated prime
ideals in the context of Leavitt path algebras. We give necessary and
sufficient conditions under which a proper ideal of a Leavitt path algebra $%
L $ is a product as well as an intersection of finitely many of these
different types of prime ideals. Such factorizations, when they are
irredundant, are shown to be unique except for the order of the factors. We
also characterize the Leavitt path algebras $L$ in which every ideal admits
such factorizations and also in which every ideal is one of these special
type of ideals.
\end{abstract}

\maketitle

\section{Introduction}

\noindent The multiplicative ideal theory in commutative algebra has been an
active area of research with contributions from many researchers including
Robert Gilmer and William Heinzer. Recently, the development of the
multiplicative ideal theory of Leavitt path algebras has become an active
area of research. Leavitt path algebras are algebraic analogues of graph
C*-algebras and are also natural generalizations of Leavitt algebras of type 
$($$1,n$$)$ constructed by William Leavitt. What stands out quite surprising
is that, even though Leavitt path algebras are highly non-commutative, the
multiplicative ideal theory of Leavitt path algebras is quite similar to
that of commutative algebras. Specifically, Leavitt path algebras satisfy a
number of characterizing properties of special types of commutative integral
domains such as the B\'{e}zout domains, the Dedekind domains, the Pr\"{u}fer
domains etc., in terms of their ideal properties (see \cite{EEKRR}, \cite%
{EMR}, \cite{R-2}). These integral domains are well-known for admitting
satisfactory ideal factorizations. Because of this, investigating the
factorizations of ideals in a Leavitt path algebra as products or
intersections of special types of ideals such as the prime, the semiprime,
the irreducible, the primary ideals is quite promising. Prime ideals and
their various generalizations play essential role in developing the
multiplicative ideal theory of commutative algebras. So it is natural to
study these notions in the case of Leavitt path algebras to develop its
multiplicative ideal theory and the various types of ideal factorizations.
The theory of prime ideals and that of semi-prime ideals for Leavitt path
algebras were developed in \cite{R-1}, \cite{AMR}. In this paper we study
three different variations in the notion of prime ideals and irreducible
ideals in the context of Leavitt path algebras.

Recall that if $P$ is a prime ideal of a ring $R$, then for any two ideals $%
A,B$ of $R$, $A\cap B\subseteq P$ implies that $A\subseteq P$ or $B\subseteq
P$. The converse of this statement is not true. For example, in the ring $%
\mathbb{Z}$ of integers, it can be verified that the ideal $8\mathbb{Z}$ has
this property by using the prime power factorization of integers in $\mathbb{%
Z}$, but $8\mathbb{Z}$ is not a prime ideal. Ideals of a ring having this
property were first studied by L. Fuchs \cite{F} who called them primitive
ideals. Blair \cite{B} called them strongly irreducible ideals. The idea was
clearly inspired by strengthening the conditions required for an ideal to be
an irreducible ideal. Recall that an ideal $I$ of a ring $R$ is said to be 
\textit{irreducible} if, for any two ideals $A,B$ of $R$, $A\cap B=I$
implies that $A=I$ or $B=I$. Interestingly, strongly irreducible ideals are
mentioned in Bourbaki's treatise on commutative algebras \cite{Bourbaki}
where they are referred to as \textit{quasi-prime} ideals. In \cite{HRR},
Heinzer, Ratliff Jr. and Rush investigated non-prime strongly irreducible
ideals of commutative noetherian rings. Recently, N. Schwartz \cite{S}
studied the truncated valuations induced by strongly irreducible ideals in
commutative rings. In general, an irreducible ideal need not be strongly
irreducible \cite{S}, however, in the case of Leavitt path algebras, our
description of the strongly irreducible ideals in Section 3 shows that these
two notions coincide. The main result in this section gives necessary and
sufficient conditions under which a proper ideal $I$ of a Leavitt path
algebra $L$ can be represented as a product as well as an intersection of
finitely many strongly irreducible ideals. Interestingly, the graded part $%
gr(I)$ of such an ideal $I$ plays an important role and, in this case, $%
I/gr(I)$ must be finitely generated. We also prove uniqueness theorems by
showing that factorizations of an ideal $I$ of $L$ as an irredundant product
or an irredundant intersection of finitely many strongly irreducible ideals
are unique except for the order of the factors (Theorems \ref{Uniqueness}
and \ref{uniqueness-intersect}). We give several characterizations (both
algebraic and graphical) of Leavitt path algebras in which every proper
ideal is uniquely an irredundant intersection/product of finitely many
strongly irreducible ideals. This answers, in the context of Leavitt path
algebras, an open question by W. Heinzer and B. Olberding (\cite{HO}) raised
for the case of commutative rings. As a by-product, we obtain a
characterization of the Leavitt path algebras which are Laskerian. We also
provide characterizations of Leavitt path algebras in which each ideal is
strongly irreducible.

It is well-known that if a prime ideal $P$ contains an intersection of
finitely many ideals $A_{i}$, $i=1,2,\ldots,n$, then $P$ contains at least
one of the ideals $A_{i}$. However, this statement fails to hold for a prime
ideal if we consider an infinite intersection of ideals. An ideal $P$ is
called \textit{strongly prime (\cite{JOT}) }if the above statement holds
even for an infinite intersection of ideals. In Section 4, we characterize
the strongly prime ideals of a Leavitt path algebra $L$ and describe when a
given ideal of $L$ can be factored as a product of strongly prime ideals. We
also describe when every ideal of $L$ admits such a factorization. In \cite%
{JOT}, the authors call a commutative ring\ $R$\textit{\ strongly
zero-dimensional} if every prime ideal of $R$ is strongly prime and they
prove a number of interesting properties of strongly zero-dimensional
commutative rings. We end the Section 4 by characterizing all the strongly
zero-dimensional Leavitt path algebras.

Recall that an arbitrary ring $R$ is a prime ring if for all $a,b\in R$,
whenever $a\neq0,b\neq0$, then there is an element $c\in R$ such that $%
acb\neq0$. In their attempts to consider the non-commutative version of
Kaplansky's conjecture on prime von Neumann regular rings, Handelman and
Lawrence \cite{HL} strengthen this concept of prime rings and consider rings
with a stronger property. They do this by restricting, for each $a\neq0$,
the choice of the $c$ to a finite set (independent of $b$, but depending on $%
a$). To make this definition precise, they define a \textit{(right) insulator%
} for $a\in R$ to be a finite subset $S(a)$ of $R$, such that the right
annihilator $ann_{R}\{ac:c\in S(a)\}=0$. A ring $R$ is said to be a \textit{%
right insulated prime ring} if every non-zero element of $R$ has a right
insulator; and an ideal $I$ of ring $R$ to be a \textit{right insulated
prime ideal }if $R/I$ is a right insulated prime ring. It is known that, in
general, the notion of insulated prime ring is not left-right symmetric. In
fact, Handelman and Lawrence constructed a ring that is right insulated
prime but not left insulated prime. In Section 5, we first describe when a
Leavitt path algebra is a left/right insulated prime ring. Interestingly,
the distinction between left and right insulated primeness vanishes for
Leavitt path algebras. We show (Theorem \ref{Type I LPA}) that a Leavitt
path algebra is a left/right insulated prime ring exactly when it is a
simple ring or it is isomorphic to the matrix ring $M_{n}(K[x,x^{-1}])$ some
integer $n\geq1$. We characterize the insulated prime ideals of Leavitt path
algebras and also describe conditions under which each ideal of a Leavitt
path algebra can be factored as a product of insulated prime ideals.

\bigskip

\section{Basics of Leavitt path algebras}

\noindent A (directed) graph $E=(E^{0}, E^{1}, r, s)$ consists of two sets $%
E^{0}$ and $E^{1}$ together with maps $r,s:E^{1}\rightarrow E^{0}$. The
elements of $E^{0}$ are called \textit{vertices} and the elements of $E^{1}$ 
\textit{edges}. A vertex $v$ is called a \textit{sink} if it emits no edges
and a vertex $v$ is called a \textit{regular} \textit{vertex} if it emits a
non-empty finite set of edges. An \textit{infinite emitter} is a vertex
which emits infinitely many edges. For each edge $e\in E^{1}$, we consider
an edge in the opposite direction, called ghost edge and denote it as $%
e^{\ast}$. So we have $r(e^{\ast})=s(e)$, and $s(e^{\ast})=r(e)$.

A\textit{\ path} $\mu$ of length $n>0$ is a finite sequence of edges $%
\mu=e_{1}e_{2}\cdot\cdot\cdot e_{n}$ with $r(e_{i})=s(e_{i+1})$ for all $%
i=1,\cdot\cdot\cdot,n-1$. We denote the length of the path $\mu$ as $|\mu|$.
Thus for $\mu=e_{1}e_{2}\cdot\cdot\cdot e_{n}$, we have $|\mu|=n$. In this
case $\mu^{\ast}=e_{n}^{\ast}\cdot\cdot\cdot e_{2}^{\ast}e_{1}^{\ast}$ is
the corresponding ghost path. A vertex is considered a path of length $0$.
The set of all vertices on a path $\mu$ is denoted as $\mu^0$. A path $\mu$ $%
=e_{1}\dots e_{n}$ in $E$ is \textit{closed} if $r(e_{n})=s(e_{1})$, in
which case $\mu$ is said to be \textit{based at the vertex }$s(e_{1})$. A
closed path $\mu$ as above is called \textit{simple} provided it does not
pass through its base more than once, i.e., $s(e_{i})\neq s(e_{1})$ for all $%
i=2,...,n$. The closed path $\mu$ is called a \textit{cycle} if it does not
pass through any of its vertices twice, that is, if $s(e_{i})\neq s(e_{j})$
for every $i\neq j$.

A graph $E$ is said to satisfy \textit{Condition $($K$)$}, if any vertex $v$
on a simple closed path $c$ is also the base of a another simple closed path 
$c^{\prime}$ different from $c$. An \textit{exit }for a path $\mu=e_{1}\dots
e_{n}$ is an edge $e$ such that $s(e)=s(e_{i})$ for some $i$ and $e\neq
e_{i} $. A graph $E$ is said to satisfy \textit{Condition $($L$)$}, if every
cycle in $E$ has an exit.

If there is a path from vertex $u$ to a vertex $v$, we write $u\geq v$. A
subset $D$ of vertices is said to be \textit{downward directed }\ if for any 
$u,v\in D$, there exists a $w\in D$ such that $u\geq w$ and $v\geq w$. When
we say that a graph $E$ is downward directed, then it means $E^{0}$ is
downward directed. A subset $H$ of $E^{0}$ is called \textit{hereditary} if,
whenever $v\in H$ and $w\in E^{0}$ satisfy $v\geq w$, then $w\in H$. A
hereditary set is \textit{saturated} if, for any regular vertex $v$, $%
r(s^{-1}(v))\subseteq H$ implies $v\in H$.

Given an arbitrary graph $E$ and a field $K$, the \textit{Leavitt path
algebra }$L_{K}(E)$ is defined to be the $K$-algebra generated by a set $%
\{v:v\in E^{0}\}$ of pair-wise orthogonal idempotents together with a set of
variables $\{e,e^{\ast}:e\in E^{1}\}$ which satisfy the following conditions:

(1) \ $s(e)e=e=er(e)$ for all $e\in E^{1}$.

(2) $r(e)e^{\ast}=e^{\ast}=e^{\ast}s(e)$\ for all $e\in E^{1}$.

(3) (The ``CK-1 relations") For all $e,f\in E^{1}$, $e^{\ast}e=r(e)$ and $%
e^{\ast}f=0$ if $e\neq f$.

(4) (The ``CK-2 relations") For every regular vertex $v\in E^{0}$, 
\begin{equation*}
v=\sum_{e\in E^{1},s(e)=v}ee^{\ast}.
\end{equation*}
Every Leavitt path algebra $L_{K}(E)$ is a $\mathbb{Z} $\textit{-graded
algebra}, namely, $L_{K}(E)={\displaystyle\bigoplus\limits_{n\in\mathbb{Z}}}
L_{n}$ induced by defining, for all $v\in E^{0}$ and $e\in E^{1}$, $\deg
(v)=0$, $\deg(e)=1$, $\deg(e^{\ast})=-1$. Here the $L_{n}$ are abelian
subgroups satisfying $L_{m}L_{n}\subseteq L_{m+n}$ for all $m,n\in\mathbb{Z}$%
. Further, for each $n\in\mathbb{Z}$, the \textit{homogeneous component }$%
L_{n}$ is given by 
\begin{equation*}
L_{n}=\{ {\textstyle\sum} k_{i}\alpha_{i}\beta_{i}^{\ast}\in L:\text{ }%
|\alpha_{i}|-|\beta_{i}|=n\}.
\end{equation*}
Elements of $L_{n}$ are called \textit{homogeneous elements}. An ideal $I$
of $L_{K}(E)$ is said to be a \textit{graded ideal} if $I=$ ${\displaystyle%
\bigoplus\limits_{n\in\mathbb{Z}}}(I\cap L_{n})$. If $A,B$ are graded
modules over a graded ring $R$, we write $A\cong_{gr}B$ if $A$ and $B$ are
graded isomorphic and we write $A\oplus_{gr}B$ to denote a graded direct
sum. We will also be using the usual grading of a matrix of finite order.
For this and for the various properties of graded rings and graded modules,
we refer to \cite{HR} and \cite{NvO}.

For any ideal $I$ of a Leavitt path algebra $L_K(E)$, $I\cap E^0$ is a
hereditary saturated subset \cite[Lemma 2.4.3]{AAS}. A \textit{breaking
vertex }of a hereditary saturated subset $H$ is an infinite emitter $w\in
E^{0}\backslash H$ with the property that $0<|s^{-1}(w)\cap
r^{-1}(E^{0}\backslash H)|<\infty$. The set of all breaking vertices of $H$
is denoted by $B_{H}$. For any $v\in B_{H}$, $v^{H}$ denotes the element $%
v-\sum_{s(e)=v,r(e)\notin H}ee^{\ast}$. Given a hereditary saturated subset $%
H$ and a subset $S\subseteq B_{H}$, $(H,S)$ is called an \textit{admissible
pair.} Given an admissible pair $(H,S)$, the ideal generated by $H\cup
\{v^{H}:v\in S\}$ is denoted by $I(H,S)$. It was shown in \cite{T} that the
graded ideals of $L_{K}(E)$ are precisely the ideals of the form $I(H,S)$
for some admissible pair $(H,S)$. We have a partial ordering on the set of
all admissible pairs of $L_K(E)$ defined as $(H_1, S_1) \leq (H_2, S_2)$ if
and only if $H_1\subseteq H_2$ and $S_1\subseteq H_2\cup S_2$. Moreover, $%
L_{K}(E)/I(H,S)\cong L_{K}(E\backslash(H,S))$. Here $E\backslash(H,S)$ is a 
\textit{Quotient graph of }$E$ where $(E\backslash(H,S))^{0}=(E^{0}%
\backslash H)\cup\{v^{\prime}:v\in B_{H}\backslash S\}$ and $%
(E\backslash(H,S))^{1}=\{e\in E^{1}:r(e)\notin H\}\cup\{e^{\prime}:e\in
E^{1} $ with $r(e)\in B_{H}\backslash S\}$ and $r,s$ are extended to $%
(E\backslash(H,S))^{1}$ by setting $s(e^{\prime})=s(e)$ and $%
r(e^{\prime})=r(e)$.

A \textit{maximal tail }is a subset $M$ of $E^{0}$ satisfying the following
three properties:

\begin{enumerate}
\item $M$ is downward directed;

\item If $u\in M$ and $v\in E^{0}$ satisfies $v\geq u$, then $v\in M$;

\item If $u\in M$ emits edges, there is at least one edge $e$ with $s(e)=u$
and $r(e)\in M$.
\end{enumerate}

A graph $E$ is called row-finite if $s^{-1}(v)$ is finite for each $v\in E^0$%
. A graph $E$ is called a comet if it is row-finite, there is a cycle $c$
without exits in $E$ and every path in $E$ ends at a vertex on $c$.

We will be using the fact that the Jacobson radical (and in particular, the
prime/Baer radical) of $L_{K}(E)$ is always zero (see \cite{AAS}).

We will make the convention that if $c$ is a cycle in the graph $E$ based at
a vertex $v$, and if $f(x)=1+k_{1}x+\cdots+k_{n}x^{n}\in K[x]$, then $%
f(c)=v+k_{1}c+\cdots+k_{n}c^{n}\in L_{K}(E)$.

In the following, ``ideal" means ``two-sided ideal" and, given a subset $S$
of $L_{K}(E)$, we shall denote by $<S>$, the ideal generated by $S$ in $%
L_{K}(E)$.

We begin with listing the various results and basic observations from the
literature about ideals in Leavitt path algebras that we will be using
throughout this paper. The next theorem describes a generating set for
ideals in a Leavitt path algebra.

\begin{theorem}
\label{genators of ideal}$($Theorem 4, \cite{R_0}$)$ Let $L_{K}(E)$ be a
Leavitt path algebra and let $I$ be an ideal of $L_{K}(E)$ with $I\cap
E^{0}=H$ and $S=\{v\in B_{H}:v^{H}\in I\}$. Then $I=I(H,S)+{\displaystyle%
\sum \limits_{t\in T}} <f_{t}(c_{t})>$ where $T$ is an index set (may be
empty) such that for each $t\in T$, $c_{t}$ is a cycle without exits in $%
E\backslash(H,S)$ and $f_{t}(x)\in K[x]$ with a non-zero constant term.
\end{theorem}

For convenience, some times we will denote $I(H,S)$ by $gr(I)$ and call it
the graded part of the ideal $I$ described above.

\noindent The next result describes the prime ideals of Leavitt path
algebras.

\begin{theorem}
\label{Prime ideals} $($Theorem 3.12, \cite{R-1}$)$ An ideal $P$ of $%
L_{K}(E) $ with $P\cap E^{0}=H$ is a prime ideal if and only if $P$
satisfies one of the following properties:

\begin{enumerate}
\item $P=I(H,B_{H})$ and $E^{0}\backslash H$ is downward directed;

\item $P=I(H,B_{H}\backslash\{u\})$, $v\geq u$ for all $v\in E^{0}\backslash
H$ and the vertex $u^{\prime}$ that corresponds to $u$ in $E\backslash
(H,B_{H}\backslash\{u\})$ is a sink;

\item $P$ is a non-graded ideal of the form $P=I(H,B_{H})+<p(c)>$, where $c$
is a cycle without exits based at a vertex $u$ in $E\backslash(H,B_{H})$, $%
v\geq u$ for all $v\in E^{0}\backslash H$ and $p(x)$ is an irreducible
polynomial in $K[x,x^{-1}]$ such that $p(c)\in P$.
\end{enumerate}
\end{theorem}

\noindent We shall also be using the following two results.

\begin{lemma}
\label{Irreducible Ideals of L} $($Theorem 5.7, \cite{R-2}$)$ If an ideal $I$
of $L_{K}(E)$ is irreducible, then $I=P^{n}$, a power of a prime ideal $P$
for some $n\geq1$. Also $gr(I)=gr(P)$ is a prime ideal.
\end{lemma}

\begin{lemma}
\label{Property of graded ideal} $($Lemma 3.1, \cite{R-2}$)$ Let $A$ be a
graded ideal of $\ L=L_{K}(E)$.

(a) For any ideal $B$ of $L$, $AB=A\cap B$; In particular, $A^{2}=A$;

(b) $A=I_{1}\ldots I_{n}$ is a product of ideals if and only if $%
A=I_{1}\cap\ldots\cap I_{n}$ is their intersection.
\end{lemma}

For a ring $R$, and an infinite set $\Lambda$, we will denote by $M_{\Lambda
}(R)$, the ring of $\Lambda\times\Lambda$ matrices in which all except at
most finitely many entries are non-zero.

For more details on results in Leavitt path algebras, we refer the reader to 
\cite{AAS} and \cite{R}.

\bigskip

\section{Strongly Irreducible Ideals of Leavitt Path Algebras}

\bigskip

\noindent In this section we describe the strongly irreducible ideals of
Leavitt path algebras. We give necessary and sufficient conditions under
which a proper ideal $I$ of a Leavitt path algebra admits a factorization as
product as well as an intertsection of finitely many strongly irreducible
ideals. Interestingly the graded part $gr(I)$ of this ideal $I$ also admits
such a factorization and in this case $I/gr(I)$ is finitely generated. We
characterize the Leavitt path algebras in which every proper ideal can be
factored as an irredundant intersection/product of finitely many strongly
irreducible ideals. Two uniqueness theorems are established showing that
such factorizations are unique except for the order of the factors. This
answers an open question of Heinzer and Olberding (\cite{HO}) in the context
of Leavitt path algebras. \ We also describe when every ideal of $L$ is
strongly irreducible. As a biproduct, Leavitt path algebras which are
Laskerian are described.

\begin{definition}
An ideal $I$ of a ring $R$ is said to be irreducible if, for any two ideals $%
A,B$ of $R$, $A\cap B=I$ implies that $A=I$ or $B=I$.
\end{definition}

\begin{definition}
An ideal $I$ of a ring $R$ is said to be a strongly irreducible ideal if,
for any two ideals $A,B$ of $R$, $A\cap B\subseteq I$ implies that $%
A\subseteq I$ or $B\subseteq I$.
\end{definition}

\noindent Clearly a prime ideal of a ring is strongly irreducible and a
strongly irreducible ideal is always irreducible. In general, an irreducible
ideal need not be strongly irreducible (see for e.g. \cite{S}, where it is
shown that in the polynomial ring $\mathbb{Q}[x,y]$, the ideal $<x,y^{2}>$
is irreducible, but not strongly irreducible) and as we have noted earlier
the ideal $8\mathbb{Z}$ is a strongly irreducible ideal but not a prime
ideal in the ring $\mathbb{Z}$ of integers. Irreducible ideals of Leavitt
path algebras are described in \cite{R-2}.

We first list some elementary (perhaps known) properties of strongly
irreducible ideals of any ring.

\medskip

(i) An ideal $I$ of a ring $R$ is strongly irreducible if, for all $a,b\in R$%
, $(aR\cap bR)\subseteq I$ (similarly, $(Ra\cap Rb)\subseteq I$) implies
that $a\in I$ or $b\in I$.

\medskip

(ii) If $I$ is a strongly irreducible ideal in $R$, then for any ideal $%
K\subseteq I$, $I/K$ is strongly irreducible in $R/K$.

\medskip

\noindent Proof of (i): Suppose $A\cap B\subseteq I$ for some ideals of $R$
and $A\nsubseteq I$. Choose $a\in A$ with $a\notin I$. Then for any $b\in B$%
, $(aR\cap bR)\subseteq A\cap B\subseteq I$ and so $b\in I$. Hence $%
B\subseteq I$.

\medskip

\noindent Proof of (ii) is straightforward.

\bigskip

\noindent Next, we list some useful results on ideals of Leavitt path
algebras over graphs containing a cycle without exits.

\begin{proposition}
\label{M=<c^0>} Suppose $c$ is a cycle without exits in a graph $E$.

$($i$)$ $\ \ ($\cite{AAS}, Theorem 2.7.1$)$ If $M$ is the ideal of $L$
generated by $c^{0}$, then $M\cong M_{\Lambda}(K[x,x^{-1}])$ for a suitable
index set $\Lambda$.

(ii) \ $($\cite{R-2}, Lemma 3.3$)$ $<f(c)><g(c)>=<f(c)g(c)>$ for any two $%
f(x),g(x)\in K[x]$. In particular, $<f(c)>^{n}=<f^{n}(c)>$ for any positive
integer $n$.

$($iii$)$ $($\cite{EEKRR}, Proposition 1$)$ The map $A\longmapsto M_{\Lambda
}(A)$ defines a lattice isomorphism between the lattices of ideals of $%
K[x,x^{-1}]$ and $M_{\Lambda}(K[x,x^{-1}])$. Moreover, $M_{\Lambda
}(AB)=M_{\Lambda}(A)M_{\Lambda}(B)$ for any two ideals \ $A,B$ of $%
K[x,x^{-1}]$.
\end{proposition}

\begin{lemma}
\label{Ideal in downward dir graph with NE cycle} Suppose $E$ is a downward
directed graph containing a cycle $c$ without exits based at a vertex $v$.
If $M$ is the ideal generated by $c^{0}$, then for every non-zero ideal $A$
of $L_{K}(E)$ either $M\subseteq A$ or $A=<f(c)>\subseteq M$ where $f(x)\in
K[x]$ with a non-zero constant term according as $A$ contains a vertex or
not.
\end{lemma}

\begin{proof}
First observe that $u\geq v$ for every vertex $u\in E$. This is because, by
downward directness of $E$, corresponding to $u,v$, there is a vertex $w$
such that $u\geq w$ and $v\geq w$. Since $v$ sits on the cycle $c$ without
exits, $w$ must be a vertex on $c$ and so $w\geq v$. Thus, $u\geq v$. So if
an ideal $A$ of $L_{K}(E)$ contains a vertex $u$, then since $A\cap E^{0}$
is hereditary, $A$ contains $v$ and hence $c^{0}$. Consequently, $%
M=<c^{0}>\subseteq A$. Suppose the non-zero ideal $A$ does not contain any
vertices. Since $E$ is downward directed, we appeal to Lemma 3.5 of \cite%
{R-1} to conclude that $A=<f(c)>$ where $f(x)\in K[x]$. In this case,
clearly $A\subseteq M$.
\end{proof}

The next theorem describes the strongly irreducible ideals of a Leavitt path
algebra $L_{K}(E)$ and shows that in the case of Leavitt path algebras, the
notions of irreducible ideals and strongly irreducible ideals coincide.

\begin{theorem}
\label{Almost prime in an LPA} \label{strongly irred} The following
properties are equivalent for an ideal $I$ of a Leavitt path algebra $%
L:=L_{K}(E)$;

\begin{enumerate}
\item $I$ is a strongly irreducible ideal of $L$;

\item $I$ is an irreducible ideal of $L$;

\item $I=P^{n}$, a power of a prime ideal $P$.
\end{enumerate}
\end{theorem}

\begin{proof}
Clearly (1) $\implies$ (2) and the implication (2) $\implies$ (3) is proved
in (\cite{R-2}, Theorem 5.7).

Assume (3), so that $I=P^{n}$ for some prime ideal $P$ and integer $n\geq 1$%
. If $P$ is graded, then by Lemma \ref{Property of graded ideal}, $I=P^{n}=P$
is a prime ideal and so is strongly irreducible. Suppose now that $P$ is a
non-graded ideal. Then, by Theorem \ref{Prime ideals}, we have $%
P=I(H,B_{H})+<p(c)>$ where $H=P\cap E^{0}$, $(E\backslash (H,B_{H}))^{0}$ is
downward directed, $c$ is a cycle without exits in $E\backslash (H,B_{H})$
and $p(x)$ is an irreducible polynomial in $K[x,x^{-1}]$. Then, using
Proposition \ref{M=<c^0>}(ii) and Lemma \ref{Property of graded ideal}, one
can show that $I=P^{n}=I(H,B_{H})+<p^{n}(c)>$. Suppose $A\cap B\subseteq I$
for some ideals $A,B$ in $L$. Let $\bar{L}=L/I(H,B_{H})\cong
L_{K}(E\backslash (H,B_{H}))$, $\bar{A}=(A+I(H,B_{H}))/I(H,B_{H})$, $\bar{B}%
=(B+I(H,B_{H}))/I(H,B_{H})$ and $\bar{I}=I/I(H,B_{H})$. Since the ideals of $%
L$ satisfy the distributive law (\cite{R-2}, Theorem 4.3), 
\begin{equation*}
(A+I(H,B_{H}))\cap (B+I(H,B_{H}))
\end{equation*}%
simplifies to $(A\cap B)+I(H,B_{H})$ and so 
\begin{equation*}
\bar{A}\cap \bar{B}=[(A\cap B)+I(H,B_{H})]/I(H,B_{H})\subseteq \bar{I}%
=<p^{n}(c)>.
\end{equation*}%
Let $M=<c^{0}>$, the ideal generated by $c^{0}$ in $\bar{L}$. By Lemma \ref%
{Ideal in downward dir graph with NE cycle}, each of $\bar{A},\bar{B}$
either contains $M$ or is contained in $M$. Now both $\bar{A}$ and $\bar{B}$
cannot contain $M$, since otherwise, $\bar{A}\cap \bar{B}\supseteq
M\varsupsetneq <p^{n}(c)>=\bar{I}$, a contradiction. If only one of them is
contained in $M$, say $\bar{A}\subseteq M$ and $\bar{B}\supseteq M$, then $%
\bar{A}=\bar{A}\cap \bar{B}\subseteq \bar{I}$ and this implies that $%
A\subseteq I$. Suppose both $\bar{A}\subseteq M$ and $\bar{B}\subseteq M$.
By Lemma \ref{Ideal in downward dir graph with NE cycle}, $\bar{A}=<f(c)>$
and $\bar{B}=<g(c)>$ where $f(x),g(x)\in K[x]$ with non-zero constant terms.
Now, by Proposition \ref{M=<c^0>}(i), $M\cong M_{\Lambda }(K[x,x^{-1}])$ for
a suitable index set $\Lambda $ and by Proposition \ref{M=<c^0>}(ii), the
ideal lattices of $M_{\Lambda }(K[x,x^{-1}])$ and the principal ideal domain 
$K[x,x^{-1}]$ are isomorphic. So, in $K[x,x^{-1}]$, 
\begin{equation*}
<f(x)>\cap <g(x)>\subseteq <p^{n}(x)>.
\end{equation*}%
Now $<f(x)>\cap <g(x)>=<h(x)>$, where $h(x)=\lcm(f(x),g(x))$.

Thus $p^{n}(x)| \lcm(f(x),g(x))$ and since $p^{n}(x)$ is a prime power, by
the uniqueness of prime power factorization in $K[x,x^{-1}]$, $p^{n}(x)|f(x)$
or $p^{n}(x)|g(x)$. This means either $<f(c)>\subseteq<p^{n}(c)>$ or $%
<g(c)>\subseteq<p^{n}(c)>$. We then conclude that either $A\subseteq I$ or $%
B\subseteq I$. This proves (1).
\end{proof}

Next, we give conditions under which a proper ideal $I$ of a Leavitt path
algebra $L_{K}(E)$ is a product as well as an intersection of finitely many
strongly irreducible ideals of $L_{K}(E)$.

\noindent We begin by proving a series of preparatory lemmas the first of
which is well-known.

\begin{lemma}
\label{Ideals in K[xx^-1]} Let $R$ be a Principal ideal domain. Then every
non-zero proper ideal $I$ of $R$ is an intersection of finitely many powers
of distinct prime ideals.
\end{lemma}

\begin{proof}
Let $I=<a>$ with $a$ ($\neq0$) being a non-unit. Let $a=p_{1}^{n_{1}}\cdots
p_{k}^{n_{k}}$ be the factorization of $a$ as a product of powers of
distinct prime (equivalently, irreducible) elements $p_{1},\ldots, p_{k}$ of 
$R$. Since $\gcd1(p_{1}^{n_{1}},\ldots, p_{k}^{n_{k}})=1$, $\lcm%
(p_{1}^{n_{1}},\ldots, p_{k}^{n_{k}})=p_{1}^{n_{1}}\cdots p_{k}^{n_{k}}$.
Consequently,%
\begin{equation*}
<p_{1}^{n_{1}}>\cap\cdots\cap<p_{k}^{n_{k}}>=<\lcm(p_{1}^{n_{1}},\ldots,
p_{k}^{n_{k}})>=<p_{1}^{n_{1}}\cdots p_{k}^{n_{k}}>=<a>\text{.}
\end{equation*}
\end{proof}

\begin{lemma}
\label{gr(I) prime=>Int. of Irred.} Suppose $I$ is a non-graded ideal of a
Leavitt path algebra $L_{K}(E)$ such that $gr(I)$ is a prime ideal. Then $I$
is an intersection of finitely many $($strongly$)$ irreducible $($= prime
power$)$ ideals.
\end{lemma}

\begin{proof}
By Theorem \ref{genators of ideal}, $I=I(H,S)+\Sigma_{t\in T}<f_{t}(c_{t})>$
where $T$ is a non-empty index set, for each $t\in T$, $c_{t}$ is a cycle
without exits in $E\backslash(H,S)$ and $f_{t}(x)\in K[x]$ with a non-zero
constant term. Since $gr(I)=I(H,S)$ is a prime ideal, by Theorem \ref{Prime
ideals}, $(E\backslash(H,S))^0$ is downward directed and so there can be
only one cycle, say $c$ without exits in $E\backslash(H,S)$. Hence we can
write $I=I(H,S)+<f(c)>$ where $c$ is then a unique cycle without exits in $%
E\backslash(H,S)$. From the description of the prime ideals in Theorem \ref%
{Prime ideals} and the fact that $I(H, S)$ is a prime ideal such that $%
E\backslash (H,S)$ has a cycle without exits, we conclude that $S=B_{H} $,
so we can write $I=I(H,B_{H})+<f(c)>$. In $\overline{L_{K}(E)}%
=L_{K}(E)/I(H,B_{H})$, $<f(c)>\subseteq M=<\{c^{0}\}>$ . Then, by
Proposition \ref{M=<c^0>} and Lemma \ref{Ideals in K[xx^-1]}, we conclude
that $<f(c)>=<p_{1}^{n_{1}}(c)>\cap \cdots\cap<p_{k}^{n_{k}}(c)> $ where $%
p_{1}(x),\ldots ,p_{k}(x)$ are distinct irreducible polynomials in $%
K[x,x^{-1}]$ and $f(x)=p_{1}^{n_{1}}(x)\cdots p_{k}^{n_{k}}(x)$ is a prime
power factorization of $f(x)$ in $K[x,x^{-1}]$. Then $I=P_{1}^{n_{1}}\cap%
\cdots\cap P_{k}^{n_{k}}$ where $P_{j}=I(H,B_{H})+<p_{j}(c)>$ is a prime
ideal for all $j=1,\ldots, k$. By Theorem \ref{Almost prime in an LPA}, each 
$P_{j}^{n_{j}}$ is a (strongly) irreducible ideal.
\end{proof}

The next technical lemma is obtained by modifying parts of the proof of
Theorem 6.2 in \cite{R-2} and is used in Theorem \ref{I intersection of
irreducibles}. Recall that given a collection of sets $\{A_{i}:i\in\mathcal{I%
}\}$, the intersection $\cap_{i\in\mathcal{I}}A_{i}$ is called irredundant,
if $\cap_{i\in\mathcal{I}\setminus\{j\}}A_{i}\nsubseteq A_{j}$ for any $j\in%
\mathcal{I}$. In particular, $A_{i}\nsubseteq A_{j}$ for any two $i\neq j\in%
\mathcal{I}$. Similarly the union $\cup_{i\in\mathcal{I}}A_{i}$ is called
irredundant if $A_{j}\nsubseteq\cup_{i\in\mathcal{I}\setminus\{j\}}A_{i}$
for any $j\in\mathcal{I}$.

\begin{lemma}
\label{Influence of gr(I0 factorization} Suppose $I=I(H,S)+\Sigma_{t\in
T}<f_{t}(c_{t})>$ is a non-graded ideal of $L$, where $T$ is a non-empty
index set, for each $t\in T$, $c_{t}$ is a cycle without exits in $%
E\backslash(H,S)$ based at a vertex $v_{t}$ with $c_{s}^{0}\cap
c_{t}^{0}=\emptyset$ for all $s,t\in T$ with $s\neq t$\ and $f_{t}(x)\in
K[x] $ with a non-zero constant term. \textit{Suppose further that }$%
I(H,S)=\cap_{j=1}^{m}P_{j}$\textit{\ is an irredundant intersection of }$m$ 
\textit{graded prime ideals }$P_{j}=I(H_{j},S_{j})$\textit{. Then} we can
take $T=\{1, \ldots, k\}$, $k\leq m$ and $I=I(H,S)+%
\Sigma_{t=1}^{k}<f_{t}(c_{t})>$. After any needed re-arrangement of indices,
we have,\textit{\ for each }$t\in T$\textit{, }$v_{t}\notin P_{t}$\textit{\
but }$v_{t}\in P_{j}$\textit{\ for all }$j=1,\ldots, m$ with $j\neq t$%
\textit{. Thus }$c_{t}\in E\backslash(H_{t},S_{t})$\textit{\ for all }$t\in
T $\textit{. }
\end{lemma}

\begin{proof}
Clearly for each $t\in T$, there is a $j_{t}$ such that $v_{t}\notin
P_{jt}=I(H_{j_{t}},S_{j_{t}})$, since, otherwise, $v_{t}\in \cap
_{j=1}^{m}P_{j}=I(H,S)$, a contradiction. We claim that, $v_{t}\in P_{j}$
for all $j\neq j_{t}$, $j=1,\ldots ,m$. Suppose, on the contrary, $%
v_{t}\notin P_{i}=I(H_{i},S_{i})$ for some $i\neq j_{t}$. Since both $%
E\backslash (H_{j_{t}},S_{j_{t}})$ and $E\backslash (H_{i},S_{i})$ contain
the cycle $c_{t}$ without exits, it is clear from the description of the
graded prime ideals in Theorem \ref{Prime ideals} that $%
P_{j_{t}}=I(H_{j_{t}},B_{H_{j_{t}}})$ and $P_{i}=I(H_{i},B_{H_{i}})$. We
wish to show that $P^{\prime }=P_{j_{t}}\cap P_{i}$ is a (graded) prime
ideal. Let $P^{\prime }=I(H^{\prime },S^{\prime })$ so that $H^{\prime
}=P^{\prime }\cap E^{0}=H_{j_{t}}\cap H_{i}$. Now $c_{t}$ is a cycle without
exits in $(E^{0}\backslash H_{j_{t}})\cup (E^{0}\backslash
H_{i})=E^{0}\backslash (H_{j_{t}}\cap H_{i})=E^{0}\backslash H^{\prime }$.
Since both $E^{0}\backslash H_{j_{t}}=(E\backslash
(H_{j_{t}},B_{H_{j_{t}}}))^{0}$ and $E^{0}\backslash H_{i}=(E\backslash
(H_{i},B_{H_{i}}))^{0}$ are downward directed, $u\geq v_{t}$ for every
vertex $u$ in $E^{0}\backslash H_{j_{t}}\cup E^{0}\backslash H_{i}$. Hence $%
E^{0}\backslash H^{\prime }=E^{0}\backslash (H_{j_{t}}\cap H_{i})=$ $%
(E^{0}\backslash H_{j_{t}})\cup (E^{0}\backslash H_{i})$ is downward
directed. We claim that \ $P^{\prime }=I(H^{\prime },B_{H^{\prime }})$. To
see this, let $u\in B_{H^{\prime }}$. We need to show that $u^{H^{\prime
}}=u-\Sigma _{e\in s^{-1}(u),r(e)\notin H^{\prime }}ee^{\ast }\in P^{\prime
} $. Noting that $ee^{\ast }\in P_{j_{t}}$ if $r(e)\in H_{j_{t}}$, we have 
\begin{equation*}
u^{H^{\prime }}=u-\Sigma _{e\in s^{-1}(u),r(e)\notin H_{j_{t}}}ee^{\ast
}-\Sigma _{e\in s^{-1}(u),r(e)\notin H^{\prime },r(e)\in H_{j_{t}}}ee^{\ast }
\end{equation*}%
\begin{equation*}
=u^{H_{j_{t}}}-\Sigma _{e\in s^{-1}(u),r(e)\notin H^{\prime },r(e)\in
H_{j_{t}}}ee^{\ast }\in P_{j_{t}}\text{.}
\end{equation*}%
By the similar argument, $u^{H^{\prime }}\in P_{i}$. Hence $u^{H^{\prime
}}\in P_{j_{t}}\cap P_{i}=P^{\prime }$. It is then clear that $P^{\prime
}=I(H^{\prime },B_{H^{\prime }})$. Since $(E\backslash (H^{\prime
},B_{H^{\prime }}))^{0}=E^{0}\backslash H^{\prime }$ is downward directed, $%
P^{\prime }$ is a prime ideal (Theorem \ref{Prime ideals}). But then $%
P_{j_{t}}\cdot P_{i}=P_{j_{t}}\cap P_{i}\subseteq P^{\prime }$ implies that $%
P_{j_{t}}\subseteq P^{\prime }$ or $P_{i}\subseteq P^{\prime }$. This
implies that $P_{j_{t}}\subseteq P_{i}$ or $P_{i}\subseteq P_{j_{t}}$
contradicting that $I(H,S)=\cap _{j=1}^{m}P_{j}$ is an irredundant
intersection. Hence, we conclude that for each $t\in T$ there is a $%
P_{j_{t}} $ such that $v_{t}\notin P_{j_{t}}$ but $v_{t}\in P_{j}$ for all $%
j\neq j_{t} $. It is also clear that if $s\in T$ with $s\neq t$ (so that $%
c_{s}^{0}\cap c_{t}^{0}=\emptyset $), then the corresponding prime ideal $%
P_{j_{s}}\neq P_{j_{t}}$. Thus the map $t\mapsto P_{j_{t}}$ is an injective
map from $T$ to $\{P_{1},\ldots ,P_{m}\}$. Consequently, $|T|\leq m$, say $%
|T|=k$. After rearranging the indices, we may assume that $T=\{1,...,k\}$
and, for each $t=1,\ldots ,k$, $v_{t}\notin P_{t}$, but $v_{t}\in P_{j}$ for
all $j=1,\ldots ,m$ with $j\neq t$. Clearly, $I=I(H,S)+\Sigma
_{t=1}^{k}<f_{t}(c_{t})>$.
\end{proof}

\begin{theorem}
\label{I intersection of irreducibles} The following properties are
equivalent for an ideal $I$ of a Leavitt path algebra $L=L_{K}(E)$:

\begin{enumerate}
\item $I$ is an intersection of finitely many (strongly) irreducible ideals;

\item $gr(I)=$\ $P_{1}\cap\cdots\cap P_{m}$ is an irredundant intersection
of graded prime ideals;

\item $gr(I)=I(H,S)=$\ $P_{1}\cap\cdots\cap P_{m}$ is an irredundant
intersection of graded prime ideals, $I=I(H,S)+%
\Sigma_{t=1}^{k}<f_{t}(c_{t})> $, where $k\leq m$, for each $t=1,\ldots, k$, 
$c_{t}$ is a cycle without exits in $E\backslash(H,S)$ based at a vertex $%
v_{t} \not\in P_t$ and $f_{t}(x)\in K[x]$ with a non-zero constant term;

\item $I$ is a product of (finitely many) strongly irreducible ideals.
\end{enumerate}
\end{theorem}

\begin{proof}
If $I$ is a graded ideal, then conditions (1) and (4) are equivalent by
Lemma \ref{Property of graded ideal}. Also, since $I=gr(I)$, conditions (1)
and (2) are easily seen to be equivalent by using Theorem \ref{strongly
irred} and Lemma \ref{Irreducible Ideals of L}. Finally, for a graded ideal,
condition (3) simplifies to condition (2).

So we may take $I$ to be a non-graded ideal.

Assume (1) so $I={\displaystyle\bigcap\limits_{j=1}^{n}}Q_{j}$ is an
intersection of (strongly) irreducible ideals of $L$. Then $gr(I)=\cap
_{j=1}^{n}gr(Q_{j})$. If needed remove appropriate ideals $gr(Q_{j})$ and,
after re-indexing, assume $gr(I)=\cap_{j=1}^{m}gr(Q_{j})$ is an irredundant
intersection. By Theorem \ref{strongly irred}, each $Q_{j}$ is a power of a
prime ideal and so, by Lemma \ref{Irreducible Ideals of L}, $gr(Q_{j})=P_{j}$
is a graded prime ideal. Thus we get a representation of $gr(I)$ as an
irredundant intersection of graded prime ideals, $gr(I)=\cap_{j=1}^{m}P_{j}$%
. This proves (2).

Assume (2) so $gr(I)=I(H,S)=\cap _{j=1}^{m}P_{j}$ is an irredundant
intersection of graded prime ideals $P_{j}$. By Lemma \ref{Influence of
gr(I0 factorization}, we then have 
\begin{equation*}
I=I(H,S)+\Sigma _{t=1}^{k}<f_{t}(c_{t})>=(P_{1}\cap \cdots \cap
P_{m})+\Sigma _{t=1}^{k}<f_{t}(c_{t})>
\end{equation*}%
where $k\leq m$, for each $t=1,\ldots ,k$, $c_{t}$ is a cycle without exits
in $E\backslash (H,S)$ based at a vertex $v_{t}$ and $f_{t}(x)\in K[x]$ with
a non-zero constant term. Moreover, for each $t=1,\ldots ,k$, $v_{t}\notin
P_{t}$, but $v_{t}\in P_{j}$ for all $j=1,\ldots ,m$ with $j\neq t$. This
proves (3).

Assume (3). For each $t=1,\ldots, k$, define $A_{t}=P_{t}+<f_{t}(c_{t})>$.
By Lemma \ref{gr(I) prime=>Int. of Irred.}, each ideal $A_{t}$ is an
intersection of finitely many (strongly) irreducible ideals of $L$. So we
are done if we show that 
\begin{equation*}
(P_{1}\cap\cdots\cap
P_{m})+\Sigma_{t=1}^{k}<f_{t}(c_{t})>=A_{1}\cap\cdots\cap A_{k}\cap
P_{k+1}\cap\cdots\cap P_{m}\text{.}
\end{equation*}
We prove this by induction on $k$. Assume $k=1$. Consider $A_{1}\cap
P_{2}\cap\cdots\cap P_{m}$. Since the ideal lattice of $L$ is distributive (%
\cite{R-2}, Theorem 4.3), we have 
\begin{align*}
A_{1}\cap P_{2}\cap\cdots\cap P_{m} &
=(P_{1}+<f_{1}(c_{1})>)\cap(P_{2}\cap\cdots\cap P_{m}) \\
& =(P_{1}\cap\cdots\cap P_{m})+<f_{1}(c_{1})>\cap(P_{2}\cap \cdots\cap P_{m})
\\
& =(P_{1}\cap\cdots\cap P_{m})+<f_{1}(c_{1})>
\end{align*}
as $c_{1}\in P_{j}$ for all $j>1$ (Lemma \ref{Influence of gr(I0
factorization}).

Suppose $k>1$ and assume that 
\begin{equation*}
A_{1}\cap\cdots\cap A_{k-1}\cap P_{k}\cap\cdots\cap
P_{m}=(P_{1}\cap\cdots\cap P_{m})+\Sigma_{t=1}^{k-1}<f_{t}(c_{t})>\text{.}
\end{equation*}
Then%
\begin{align*}
& A_{1}\cap\cdots\cap A_{k}\cap P_{k+1}\cap\cdots\cap P_{m} \\
& =A_{1}\cap\cdots\cap A_{k-1}\cap(P_{k}+<f_{k}(c_{k})>)\cap
P_{k+1}\cap\cdots\cap P_{m} \\
& =(A_{1}\cap\cdots\cap A_{k-1}\cap P_{k}\cap\cdots\cap
P_{m})+A_{1}\cap\cdots\cap A_{k-1}\cap<f_{k}(c_{k})>\cap
P_{k+1}\cap\cdots\cap P_{m} \\
& =(P_{1}\cap\cdots\cap
P_{m})+\Sigma_{t=1}^{k-1}<f_{t}(c_{t})>+<f_{k}(c_{k})>
\end{align*}
\newline
as $<f_{k}(c_{k})>\subseteq P_{j}$ for all $j \neq k$ (Lemma \ref{Influence
of gr(I0 factorization})

By induction, we conclude that $I$ is an intersection of finitely many
(strongly) irreducible ideals. This proves (1).

Finally, the equivalence of conditions (4) and (2) follows from the fact
that a product of strongly irreducible ideals is also a product of prime
ideals (as a strongly irreducible ideal is a power of a prime ideal by
Theorem \ref{strongly irred}) and that the equivalence of condition (2) with
the existence of the prime factorization of $I$ is established in Theorem
6.2 of \cite{R-2}.
\end{proof}

\begin{remark}
\textrm{In general, a product of finitely many distinct prime ideals in a
general ring $R$ need not be equal to their intersection. We use the example
from \cite{RM}. Let $R=K[x,y]$ be the commutative ring of polynomials in two
variables $x,y$ over a field $K$. Now $P=<x,y>$ and $Q=<x>$ are prime ideals
of $R$. Then $PQ=<x^{2},xy>$ and $P\cap Q=Q$. Clearly, $PQ\neq P\cap Q$. In
contrast, the preceding theorem states that an ideal $I$ of a Leavitt path
algebra is a product of (finitely many) strongly irreducible ideals if and
only if $I$ is an intersection of finitely many strongly irreducible ideals. 
}
\end{remark}

Our next goal is to prove the uniqueness of factorization of an ideal of $L$
as a product as well as an intersection of finitely many strongly
irreducible ideals. Since strongly irreducible ideals are powers of prime
ideals, we consider the uniqueness of representing an ideal of $L$ as a
product/intersection of powers of distinct prime ideals. This is done in the
next two theorems. Here, we say a product $I=A_{1}\cdots A_{k}$ of distinct
ideals $A_{j}$ is an \textit{irredundant product} if $I$ is not the product
of any proper subset of ideals of the set $\{A_{1},\ldots ,A_{k}\}$.
Likewise, an intersection of distinct ideals $I=A_{1}\cap \cdots \cap A_{m}$
is said to be \textit{irredundant} if $I$ is not the intersection of any
proper subset of $\{A_{1},\ldots ,A_{m}\}$. In the proof of the theorem, we
shall be using the following properties of prime ideals of $L$.

\begin{lemma}
\label{Properties of primes} (a) Let $P$ be a prime ideal of $L$ and let $A$
be any ideal with $P\subseteq A$. Then

(i) \ (Lemma 5.3, \cite{R-2}) Either $P\subseteq gr(A)$ or $P=A$;

(ii) (Corollary 4.5, \cite{R-2}) If $P\neq A$, then $P=AP$.

(b) (Lemma 5.2, \cite{R-2}) Suppose $P$ is a prime ideal of $L$ and $A$ is
an ideal such that $P^{n}\subseteq A\subseteq P$ for some $n>1$. Then $%
A=P^{r}$ for some $1\leq r\leq n$;
\end{lemma}

\begin{theorem}
\label{Uniqueness} Suppose 
\begin{equation*}
I=P_{1}^{r_{1}}\cdots P_{m}^{r_{m}}=Q_{1}^{s_{1}}\cdots Q_{n}^{s_{n}}
\end{equation*}
are two representations of an ideal $I$ of a Leavitt path algebra $L$ as an
irredundant product of powers of distinct prime ideals. Then $m=n$ and $%
\{P_{1}^{r_{1}},\ldots, P_{m}^{r_{m}}\}=\{Q_{1}^{s_{1}},\ldots,
Q_{n}^{s_{n}}\}$.
\end{theorem}

\begin{proof}
Now the prime ideal $P_{1}$ contains the product $Q_{1}^{s_{1}}\cdots
Q_{n}^{s_{n}}$ and so $P_{1}\supseteq Q_{j_{1}}$ for some index $j_{1}$. In
the same way, the prime ideal $Q_{j_{1}}$ contains $P_{1}^{r_{1}}\cdots
P_{m}^{r_{m}}$ and so $Q_{j_{1}}\supseteq P_{i_{1}}$ for some $i_{1}$. So $%
P_{1}\supseteq P_{i_{1}}$. We claim that $P_{1}=P_{i_{1}}$. Because, $%
P_{1}\supsetneq P_{i_{1}}$ implies, by Lemma \ref{Properties of primes}%
(a)(i), that $gr(P_{1})\supseteq P_{i_{1}}$and so $%
gr(P_{1})=(gr(P_{1}))^{r_{i_{1}}}\supseteq P_{i_{1}}^{r_{i_{1}}}$ which
implies that $gr(P_{1})P_{i_{1}}^{r_{i_{1}}}=gr(P_{1})\cap
P_{i_{1}}^{r_{i_{1}}}=P_{i_{1}}^{r_{i_{1}}}$. By Lemma \ref{Property of
graded ideal}, we have $%
P_{1}^{r_{1}}P_{i_{1}}^{r_{i_{1}}}=P_{1}^{r_{1}}(gr(P_{1})P_{i_{1}}^{r_{i_{1}}})=(P_{1}^{r_{1}}\cap gr(P_{1}))P_{i_{1}}^{r_{i_{1}}}=gr(P_{1})P_{i_{1}}^{r_{i_{1}}}=P_{i_{1}}^{r_{i_{1}}} 
$. Then, in the product $P_{1}^{r_{1}}\cdots P_{m}^{r_{m}}$, using the
commutativity of the ideal multiplication (\cite{R-2}, Theorem 3.4; \cite%
{AAS}, Corollary 2.8.17), $P_{1}^{r_{1}}P_{i_{1}}^{r_{i_{1}}}$ can be
replaced by $P_{i_{1}}^{r_{i_{1}}}$. This contradicts the irredundancy of
the product. Thus $P_{1}=P_{i_{1}}$ and consequently, $P_{1}=Q_{j_{1}}$.
Re-arranging the factors, we write, without loss of generality, 
\begin{equation*}
I=P_{1}^{r_{1}}P_{2}^{r_{2}}\cdots
P_{m}^{r_{m}}=P_{1}^{s_{j_{1}}}Q_{2}^{s_{2}}\cdots Q_{n}^{s_{n}}.
\end{equation*}
Repeating this process, using the irredundancy and successively replacing $%
Q_{j_{2}},\ldots, Q_{j_{m}}$ by $P_{2},\ldots, P_{m}$ respectively, we get $%
m\leq n$. Likewise, starting with the prime ideals $Q_{1},\ldots, Q_{n}$ and
replacing them by the ideals $P_{i_{1}},\ldots, P_{i_{n}}$ we conclude that $%
n\leq m$. Consequently $m=n$, the map $j\mapsto i_{j}$ is a permutation $%
\sigma$ on the set $\{1,\ldots,m\}$ such that $P_{j}=Q_{\sigma(j)}$. In
particular, $\{P_{1},\ldots, P_{m}\}=\{Q_{1}, \ldots, Q_{m}\}$. Thus%
\begin{equation*}
I=P_{1}^{r_{1}}P_{2}^{r_{2}}\cdots
P_{m}^{r_{m}}=P_{1}^{t_{i}}P_{2}^{t_{2}}\cdots
P_{m}^{t_{m}}\qquad\qquad\qquad(\ast)
\end{equation*}
is an irredundant product of powers of distinct prime ideals where $%
t_{j}=s_{\sigma(j)}$ for all $j=1,\cdot\cdot\cdot,m$ and where we assume
that $r_{j}=1=t_{j}$ if $P_{j}$ is a graded ideal (Lemma \ref{Property of
graded ideal}). Note that necessarily $P_{i}\nsubseteq P_{j}$ for all $i\neq
j$; because, $P_{i}\subseteq P_{j}$ and $P_{i}\neq P_{j}$ implies that $%
P_{i} $ $\subseteq gr(P_{j})$ and $P_{i}=P_{i}P_{j}$ (Lemma \ref{Properties
of primes} (a)(i),(ii)). As before, we can then derive that $%
P_{i}^{r_{i}}P_{j}^{r_{j}}=P_{i}^{r_{i}}$\ and this will lead to
contradicting the irredundancy of the product (*)

Next, we wish to show that the exponents $r_{j}=t_{j}$ for all $j=1,\ldots,
m $. If all the ideals $P_{j}$ are graded, then $%
P_{j}^{r_{j}}=P_{j}=P_{j}^{t_{j}}$ and we can conclude that $r_{j}=1=t_{j}$
for all $j=1,\ldots, m$. So assume that at least one of the ideals is
non-graded. Let $j$ be an arbitrary index for which $P_{j}$ is a non-graded
(prime) ideal. So we are done if we show that $r_{j}=t_{j}$. Using the
commutativity of the ideal multiplication in $L$ and re-indexing, we may
assume, for convenience in writing, that $j=1$ and so $P_{1}$ is a
non-graded prime ideal, say $P_{1}=I(H,B_{H})+<p(c)>$, where $H=P_{1}\cap
E^{0}$, $E\backslash(H,B_{H})$ is downward directed, $c$ is a cycle without
exits in $E\backslash(H,B_{H})$ based at a vertex $v$ and $p(x)$ is an
irreducible polynomial in $K[x,x^{-1}]$. Now $\bar{L}=L/I(H,B_{H})\cong
L_{K}(E\backslash(H,B_{H}))$ and under this isomorphism, we identify $\bar{L}
$ with $L_{K}(E\backslash(H,B_{H}))$. Let $\bar{P}%
_{j}=(P_{j}+I(H,B_{H}))/I(H,B_{H})$ for all $j=1,\ldots, m$ and let $M$ be
the ideal of $\bar{L}$ generated by $c^{0}$. Clearly $\bar{P}%
_{1}=<p(c)>\varsubsetneq M$. By Lemma \ref{Ideal in downward dir graph with
NE cycle}, we have, for each $j\geq2$, \ either $M\subseteq\bar{P}_{j}$ or $%
\bar{P}_{j}\subseteq M$ according as $\bar{P}_{j}$ contains a vertex or not.

If $M\subseteq\bar{P}_{j}$ $\ $for all $j=2,\ldots, m$ then, since $M$ is a
graded ideal, we have, by Lemma 2.4, $M\bar{P}_{j}=M\cap\bar{P}_{j}=M$ for
all $j=2,\ldots, m$ and also $\bar{P}_{1}M=\bar{P}_{1}\cap M=\bar{P}_{1}$.\
Using these equations repeatedly, we then have, 
\begin{equation*}
\bar{P}_{1}^{r_{1}}\bar{P}_{2}^{r_{2}}\cdots\bar{P}_{m}^{r_{m}}=\bar{P}%
_{1}^{r_{1}}M\bar{P}_{2}^{r_{2}}\cdots\bar{P}_{m}^{r_{m}}=\bar{P}%
_{1}^{r_{1}}M=\bar{P}_{1}^{r_{1}}=<p^{r_{1}}(c)>.
\end{equation*}
Likewise, 
\begin{equation*}
\bar{P}_{1}^{t_{1}}\bar{P}_{2}^{t_{2}}\cdots\bar{P}_{m}^{t_{m}}=\bar{P}%
_{1}^{t_{1}}M\bar{P}_{2}^{t_{2}}\cdots\bar{P}_{m}^{t_{m}}=\bar{P}%
_{1}^{t_{1}}M=\bar{P}_{1}^{t_{1}}=<p^{t_{1}}(c)>.
\end{equation*}
From the equation (*), we have $<p^{r_{1}}(c)>=<p^{t_{1}}(c)\subseteq
vL_{K}(E\backslash(H,B_{H}))v$, noting that $vp(c)v=p(c)$. Now $%
vL_{K}(E\backslash(H,B_{H}))v\overset{\theta}{\cong}K[x,x^{-1}]$ where the
isomorphism $\theta$ maps $v$ to $1$, $c$ to $x$, $c^{\ast}$ to $x^{-1}$ and
thus maps $p(c)$ to $p(x)$. Hence $<p^{r_{1}}(c)>=<p^{t_{1}}(c)>$ implies
that $<p^{r_{1}}(x)>=<p^{t_{1}}(x)>$ in $K[x,x^{-1}]$. Since $K[x,x^{-1}]$
is a unique factorization domain and $p(x)$ is irreducible, we then conclude
that $r_{1}=t_{1}$.

Suppose not all the $\bar{P}_{j}$ contain $M$. Without loss of generality,
assume $\bar{P}_{j}\subsetneq M$ for $j=2,\ldots, k$ and $\bar{P}%
_{j}\supseteqq M$ \ for $j=k+1,\ldots, m$. By Lemma \ref{Ideal in downward
dir graph with NE cycle}, $\bar{P}_{j}^{r_{j}}=<f_{j}(c)>$ for all $%
j=2,\ldots, k$ where $f_{j}(x)\in K[x]$ with a non-zero constant term and
moreover, $\bar{P}_{j}^{r_{j}}$ does not contain any vertex in $\bar{L}$.
This means that $gr(P_{j})=gr(P_{j}^{r_{j}})\subseteq I(H,B_{H})$. Notice
that this implies that the ideal $P_{j}$ is not graded (since otherwise, $%
P_{j}=gr(P_{j})\subseteq I(H,B_{H})\varsubsetneq P_{1}$ which implies, by
Lemma \ref{Properties of primes}, that $P_{j}=P_{j}P_{1}$ and this leads to
contradicting the irredundancy of the product (*)). Let $%
P_{j}=I(H_{j},B_{H_{j}})+<q_{j}(c_{j})>$ where $c_{j}$ is a cycle without
exits in $E\backslash(H_{j},B_{H_{j}})$ which is downward directed and $%
q_{j}(x)$ is an irreducible polynomial in $K[x,x^{-1}]$. We claim that $%
H_{j}=H$. Otherwise, there will be a vertex $u\in H\backslash H_{j}$ and, as 
$u\geq w$ for some vertex $w\in c_{j}^{0}$ (due to downward directness), $H$
will contain $w$ and hence $c_{j}^{0}$. This will imply that $P_{1}\supseteq
P_{j}$, a contradiction as $P_{i}\nsubseteq P_{j}$ for all $i\neq j$. Thus $%
E\backslash(H,B_{H})=E\backslash(H_{j},B_{H_{j}})$, $c_{j}=c$ and $%
P_{j}=I(H,B_{H})+<q_{j}(c)>$. This holds for all $j=2,\ldots, k$. Clearly, $%
p(x)\neq q_{j}(x)$ for any $j=2,\ldots, k$. Now, as noted in the preceding
paragraph, $M\bar{P}_{k+1}^{r_{k+1}}\cdots\bar{P}_{m}^{r_{m}}=M$ and that $%
\bar{P}_{1}^{r_{1}}\cdots\bar{P}_{k}^{r_{k}}M=\bar{P}_{1}^{r_{1}}\cdots\bar{P%
}_{k}^{r_{k}}$ and so we have 
\begin{align*}
\bar{P}_{1}^{r_{1}}\bar{P}_{2}^{r_{2}}\cdots\bar{P}_{m}^{r_{m}} & =\bar{P}%
_{1}^{r_{1}}\cdots\bar{P}_{k}^{r_{k}}M\bar{P}_{k+1}^{r_{k+1}}\cdots\bar{P}%
_{m}^{r_{m}}=\bar{P}_{1}^{r_{1}}\cdots\bar {P}_{k}^{r_{k}} \\
& =<p^{r_{1}}(c)><q_{2}^{r_{2}}(c)>\cdots<q_{k}^{r_{k}}(c)>.
\end{align*}
Similarly, 
\begin{align*}
\bar{P}_{1}^{t_{1}}\bar{P}_{2}^{t_{2}}\cdots\bar{P}_{m}^{t_{m}} & =\bar{P}%
_{1}^{t_{1}}\cdots\bar{P}_{k}^{t_{k}}M\bar{P}_{k+1}^{t_{k+1}}\cdots\bar{P}%
_{m}^{t_{m}}=\bar{P}_{1}^{t_{1}}\cdots\bar {P}_{k}^{t_{k}} \\
& =<p^{t_{1}}(c)><q_{2}^{t_{2}}(c)>\cdots<q_{k}^{t_{k}}(c)>.
\end{align*}
From the equation (*), we have 
\begin{equation*}
<p^{r_{1}}(c)><q_{2}^{r_{2}}(c)>%
\cdots<q_{k}^{r_{k}}(c)>=<p^{t_{1}}(c)><q_{2}^{t_{2}}(c)>%
\cdots<q_{k}^{t_{k}}(c)>.
\end{equation*}
Again, as noted in the previous paragraph, we use the isomorphism%
\begin{equation*}
vL_{K}(E\backslash(H,B_{H}))v\overset{\theta}{\cong}K[x,x^{-1}]
\end{equation*}
to conclude that, in $K[x,x^{-1}]$ 
\begin{equation*}
<p^{r_{1}}(x)><q_{2}^{r_{2}}(x)>%
\cdots<q_{k}^{r_{k}}(x)>=<p^{t_{1}}(x)><q_{2}^{t_{2}}(x)>%
\cdots<q_{k}^{t_{k}}(x)>
\end{equation*}
Now $p(x),q_{2}(x)$,$\ldots, q_{k}(x)$ \ are all distinct irreducible
polynomials in $K[x,x^{-1}]$ and so, by the uniqueness of prime power
factorization in $K[x,x^{-1}]$, we conclude that $r_{1}=t_{1}$. By repeating
this argument for every $j$ for which $P_{j}$ is a non-graded ideal, we
conclude that $r_{j}=t_{j}$ for all $j=1,\ldots, m$. This proves Theorem \ref%
{Uniqueness}.
\end{proof}

\noindent The next theorem considers the uniqueness of representing an ideal 
$I$ of $L$ as an irredundant intersection of strongly irreducible ideals.
Again, we state the theorem in terms of powers of distinct prime ideals. It
states, in particular, that if $I$ is an irredundant intersection of
finitely many powers of distinct prime ideals, then $I$ cannot be an
irredundant intersection of infinitely powers of distinct prime ideals.

\begin{theorem}
\label{uniqueness-intersect} Suppose 
\begin{equation*}
I=P_{1}^{r_{1}}\cap\cdots\cap P_{m}^{r_{m}}=\bigcap \limits_{j\in
J}Q_{j}^{s_{j}}
\end{equation*}
are two irredundant intersections of an ideal $I$ of $L$, where $J$ is an
arbitrary index set and the ideals $P_{i}^{r_{i}},Q_{j}^{s_{j}}$ are powers
of distinct prime ideals. Then $|J|=m$ and $\{P_{1}^{r_{1}},\ldots
,P_{m}^{r_{m}}\}=\{Q_{j}^{s_{j}}:j\in J\}$.
\end{theorem}

\begin{proof}
The idea of the proof is essentially the one used in the proof of Theorem %
\ref{Uniqueness}. After noting that \textrm{\ }$P_{1}P_{i_{1}}=P_{i_{1}}$
implies that $P_{1}\cap P_{i_{1}}=P_{i_{1}}$ and likewise, $%
P_{i}^{r_{i}}P_{j}^{r_{j}}=P_{i}^{r_{i}}$ implies that $P_{i}^{r_{i}}\cap
P_{j}^{r_{j}}=P_{i}^{r_{i}}$, proceed as in the proof of Theorem \ref%
{Uniqueness} to conclude that for each $i=1,\ldots, m$, $P_{i}=Q_{j_{i}}$.
We claim that $|J|=m$. Suppose not. Let $J^{\ast
}=J\backslash\{j_{1},\ldots, j_{m}\}$. Consider the irredundant
intersections 
\begin{equation*}
I=P_{1}^{r_{1}}\cap\cdots\cap P_{m}^{r_{m}}=(P_{1}^{s_{1}}\cap \cdots\cap
P_{m}^{s_{m}})\cap(\bigcap \limits_{j\in
J^{\ast}}Q_{j}^{s_{j}})\qquad\qquad\qquad(\#).
\end{equation*}
Now, for a $j\in J^{\ast}$, the corresponding prime ideal $Q_{j}$ contains $%
P_{1}^{r_{1}}\cap\cdots\cap P_{m}^{r_{m}}$ and so $Q_{j}\supseteq P_{i}$ for
some $i\in\{1,\ldots, m\}$. Note that $Q_{j}\neq P_{i}$, because otherwise $%
Q_{j}^{s_{j}}=P_{i}^{s_{j}}$ and $P_{i}^{s_{i}}\cap
Q_{j}^{s_{j}}=P_{i}^{s_{i}}\cap P_{i}^{s_{j}}=P_{i}^{s_{i}}$ or $%
Q_{j}^{s_{j}}$ leading to contradicting the irredundancy of the intersection 
$(P_{1}^{s_{1}}\cap\cdots\cap P_{m}^{s_{m}})\cap(\bigcap \limits_{j\in
J^{\ast}}Q_{j}^{s_{j}})$. Appealing to Lemma \ref{Properties of primes}%
(a)(ii), we then get $P_{i}\subseteq gr(Q_{j})$ which implies that $%
P_{i}^{s_{i}}\subseteq(gr(Q_{j}))^{s_{i}}=$ $gr(Q_{j})\subseteq
Q_{j}^{s_{j}} $. Then $P_{i}^{s_{i}}\cap Q_{j}^{s_{j}}=P_{i}^{r_{i}}$ which
again contradicts the irredundancy of the intersection \ $%
(P_{1}^{s_{1}}\cap\cdots\cap P_{m}^{s_{m}})\cap(\bigcap \limits_{j\in
J^{\ast}}Q_{j}^{s_{j}})$. Thus $|J|=m$ and we have the irredundant
intersections 
\begin{equation*}
I=P_{1}^{r_{1}}\cap\cdots\cap P_{m}^{r_{m}}=Q_{j_{1}}^{s_{1}}\cap\cdots\cap
Q_{j_{m}}^{s_{m}}=P_{1}^{s_{1}}\cap\cdots\cap P_{m}^{s_{m}}.
\end{equation*}
Then proceed as in the proof of Theorem \ref{Uniqueness} to reach the
conclusion that $r_{j}=s_{j}$ for all $j=1,\ldots,m$ after noting that, in
the unique factorization domain $K[x,x^{-1}]$, $<p^{k}(x)>\cap
<q^{n}(x)>=<p^{k}(x)><q^{n}(x)>$ for any two irreducible polynomials $p(x)$
and $q(x)$.
\end{proof}

In \cite{HO}, Heinzer and Olberding raised the following question which
(according to Bruce Olberding) is still open.

\begin{question}
$($Heinzer - Olberding, \cite{HO}$)$ Under what conditions every ideal in a
commutative ring $R$ can be uniquely represented as an irredundant
intersection of irreducible ideals?
\end{question}

We completely answer this question in the context of Leavitt path algebras.
Interestingly, it turns out (see Theorem \ref{Uniqueness}) that, in a
Leavitt path algebra, if an ideal $I$ is represented as an irredundant
intersection/product of finitely many irreducible ideals, then such a
representation is automatically unique.

\begin{theorem}
\label{Every I is int irreducibles}The following properties are equivalent
for a Leavitt path algebra $L=L_{K}(E)$:

\begin{enumerate}
\item Every ideal of $L$ is an irredundant intersection of finitely many
(strongly) irreducible ideals;

\item Every ideal of $L$ is a product of (finitely many) strongly
irreducible ideals;

\item $L$ is a generalized ZPI ring, that is, every ideal of $L$ is a
product of prime ideals;

\item Every non-zero homomorphic image of $L$ is either a prime ring or
contains only finite number of minimal prime ideals;

\item For every admissible pair $(H,S)$, $(E\backslash(H,S))^{0}$ is the
union of a finite number of maximal tails.
\end{enumerate}
\end{theorem}

\begin{proof}
Now (1) $\Longleftrightarrow$ (2) follows from the equivalence of conditions
(1) and (4) in Theorem \ref{I intersection of irreducibles}.

Assume (2). Let $I$ be an arbitrary ideal of $L$. We are given that $%
I=Q_{1}\cdots Q_{n}$ is a product of strongly irreducible ideals. By Theorem %
\ref{strongly irred}, each $Q_{j}=P_{j}^{k_{j}}$, where the $P_{j}$ are
prime ideals with $k_{j}\geq1$,. Expanding each $P_{j}^{k_{j}}$, $I$ then
becomes a product of prime ideals. Thus $L$ is a generalized ZPI ring. This
proves (3).

Since every prime ideal is strongly irreducible, condition (3) implies
condition (2).

The equivalence of conditions (3) and (4) has been established in (Theorem
6.5, \cite{R-2}).

Assume (1). Given an admissible pair $(H,S)$, consider the graded ideal $%
A=I(H,S)$. By hypothesis, $A={\displaystyle\bigcap\limits_{j=1}^{m}} A_{j}$
where each $A_{j}$ is strongly irreducible and $m$ is a positive integer.
Since $A$ is graded, $A={\displaystyle\bigcap\limits_{j=1}^{m}} gr(A_{j})={%
\displaystyle\bigcap\limits_{j=1}^{m}} P_{j}$ where $P_{j}=gr(A_{j})$ is a
(graded) prime ideal as $A_{j}$ is a power of a prime ideal by Theorem \ref%
{strongly irred}. Then, in $L_{K}(E\backslash(H,S))\cong L/I(H,S)$, $\{0\}={%
\displaystyle\bigcap\limits_{j=1}^{m}} Q_{j}$ where each $Q_{j}=P_{j}/I(H,S)$
is a graded prime ideal, say $Q_{j}=I(H_{j},S_{j})$ where $%
H_{j}=Q_{j}\cap(E\backslash(H,S))^{0}$ and $(E\backslash(H,S))^{0}\backslash
H_{j}$ is downward directed. Now ${\displaystyle\bigcap\limits_{j=1}^{m}}
H_{j}=\emptyset$ and so $(E\backslash(H,S))^{0}={\displaystyle\bigcup
\limits_{j=1}^{m}} M_{j}$ where $M_{j}=(E\backslash(H,S))^{0}\backslash
H_{j} $ is a maximal tail. This proves (5).

Assume (5). We shall prove (1). In view of Theorem \ref{I intersection of
irreducibles}(2), it is enough if we show that every graded ideal of $L$ is
an intersection of finitely many graded prime ideal. This is done just by
reversing the arguments in the proof of (1) $\implies$ (5). Let $A=I(H,S)$
be a graded ideal. By hypothesis, $(E\backslash (H,S))^{0}={\displaystyle%
\bigcup\limits_{j=1}^{n}} M_{j}$ is a union of $n$ maximal tails for
positive integer $n$. For each $j=1,\ldots,n$, define $H_{j}=(E%
\backslash(H,S))^{0}\backslash M_{j}$ and $Q_{j}=I(H_{j},B_{H_{j}})$. Then
each $Q_{j}$ is a graded prime ideal and ${\displaystyle\bigcap%
\limits_{j=1}^{n}}Q_{j}=\{0\}$, as ${\displaystyle\bigcap\limits_{j=1}^{n}}
H_{j}=\emptyset$. It is then clear that $A=I(H,S)={\displaystyle%
\bigcap\limits_{j=1}^{n}} P_{j}$ where $P_{j}\supseteq I(H,S)$ is a graded
prime ideal such that $P_{j}/I(H,S)=Q_{j}$. This proves (1).
\end{proof}

\noindent The example below illustrates the above theorem for a finite graph.

\begin{example}
\textrm{Let $E$ be a finite graph, or more generally, let $E^{0}$ be finite.
Then every ideal of the Leavitt path algebra $L_{K}(E)$ is an intersection
of finitely many strongly irreducible ideals. Justification: In view of
Theorem \ref{Every I is int irreducibles}, we need only to show that every
graded ideal $I$ of $L$ is an intersection of finitely many prime ideals of $%
L$. Let $I\cap E^{0}=H$. Since $L/I\cong L_{K}(E\backslash H)$ is a Leavitt
path algebra, its prime radical is $0$, and so $I=\cap\{P:P$ prime ideal $%
\supseteq I\}$ = $\cap\{gr(P):P$ graded prime ideal $\supseteq I\}$. Since $%
(E\backslash H)^{0}$ is finite, there are only finitely many hereditary
saturated subsets of $(E\backslash H)^{0}$ and so there are only finitely
many graded (in particular, graded prime) ideals in $L/I$. This means that $%
I $ is an intersection of finitely many graded prime ideals. }
\end{example}

\noindent To illustrate Theorem \ref{Every I is int irreducibles} for an
infinite graph, let us consider the following example.

\begin{example}
\label{row-finite/countable graph} \rm Let $E$ be the following graph

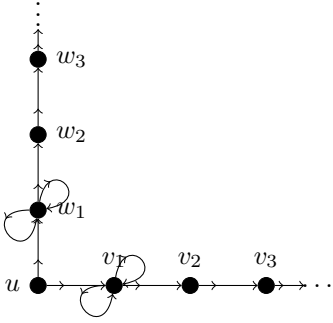
\begin{tikzpicture}[x=1cm, y=1cm,every edge/.style={draw, postaction={decorate,decoration={markings,mark=at position .3 with {\arrow{>}}}}}]
\node[vertex, label = {left: $u$}] (u) at (-3, 1) {}; 
\node[vertex, label = {above: $v_1$}] (v1) at (-2,1) {};
\node[vertex, label = {above: $v_2$}] (v2) at (-1, 1) {}; 
\node[vertex, label = {above: $v_3$}] (v3) at (0, 1) {};  
\node at (.7,1) {\dots};  
\node[vertex, label = {right: $w_1$}] (w1) at (-3,2) {};
\node[vertex, label = {right: $w_2$}] (w2) at (-3, 3) {}; 
\node[vertex, label = {right: $w_3$}] (w3) at (-3, 4) {}; 
\node at (-3,4.7) {\vdots};   
\path[->] (u)  edge (v1);
\path[->] (v1)  edge (v2);
\path[->] (v2)  edge  (v3);
\path[->] (v3)  edge (0.5,1);  
\path[->] (u)  edge (w1);
\path[->] (w1)  edge (w2);
\path[->] (w2)  edge  (w3);
\path[->] (w3)  edge (-3,4.4);

\path [->] (w1) edge [in=10 ,out=80, looseness=13,->](w1);
\path [->] (w1) edge [in=260 ,out=180,looseness=13,->](w1);
\path [->] (v1) edge [in=10 ,out=80, looseness=13,->](v1);
\path [->] (v1) edge [in=260 ,out=180,looseness=13,->](v1);

\end{tikzpicture}

\noindent The graph is row-finite and satisfies the Condition (K).
The proper non-empty hereditary saturated subsets of this graph are
precisely the sets $A_{1}=\{v_{n}:n\geq 1\}$; $A_{2}=\{v_{n}:n\geq 2\}$; $%
B_{1}=\{w_{n}:n\geq 1\}$; $B_{2}=\{w_{n}:n\geq 2\}$. For each $i=1,2$, $%
E^{0}\setminus A_{i}$ and also $E^{0}\setminus B_{i}$ is the union of at
most two maximal tails. Thus every non-zero ideal of $L_{K}(E)$ is the
intersection as well as a product of at most two graded strongly irreducible
ideals. If $\ P_{1}=<A_{1}>$ and $P_{2}=<B_{1}>$, then $P_{1}$ and $P_{2}$
are prime and hence strongly irreducible ideals and $P_{1}\cap P_{2}=\{0\}$.
Thus every ideal of $L_{K}(E)$ is a product/intersection of at most two
strongly irreducible ideals. 
\end{example}

\begin{remark}
\textrm{Recall, an ideal $I$ of a ring $R$ is said to be a\textit{\ primary
ideal} if, for all ideals $A,B$ of $R$, $AB\subseteq I$ and $A\nsubseteq I$
implies that $B\subseteq rad(I)$. A ring $R$ is said to be \textit{Laskerian}
(or simply, \textit{Lasker}) if every ideal of $R$ is an intersection of
finitely many primary ideals. It was shown in Theorem 5.7 of \cite{R-2},
that an ideal $I$ is a primary ideal of $L$ if and only if $I$ is a power of
a prime ideal which, by Theorem 3.5, is equivalent to being strongly
irreducible. Thus Theorem \ref{Every I is int irreducibles} gives a complete
description of Leavitt path algebras which are Laskerian. }
\end{remark}

Next we consider when every ideal of a Leavitt path algebra $L_{K}(E)$ is
strongly irreducible.

\begin{theorem}
\label{Every Id is str. Irred}The following are equivalent for any Leavitt
path algebra $L=L_{K}(E)$:

\begin{enumerate}
\item Every ideal of $L$ is strongly irreducible;

\item All the ideals of $L$ are graded and form a chain under set inclusion;

\item The graph $E$ satisfies Condition (K) and the admissible pairs $(H,S)$
form a chain under the partial ordering of the admissible pairs;

\item Every ideal of $L$ is a prime ideal.
\end{enumerate}
\end{theorem}

\begin{proof}
Assume (1). Let $A,B$ be any two ideals of $L$. Now $I=A\cap B$ is strongly
irreducible and hence irreducible. So $I=A$ or $I=B$. This means $A=A\cap
B\subseteq B$ or $B=A\cap B\subseteq A$. Thus ideals of $L$ form a chain
under set inclusion. Suppose, by way of contradiction, that $L$ contains a
non-graded ideal $J$. By Theorem \ref{genators of ideal}, $J=I(H,S)+\Sigma
_{i\in X}<f_{i}(c_{i})>$, where $X$ is a non-empty index set, for each $i\in
X$, $c_{i}$ is a cycle without exits in $E\backslash(H,S)$ based at a vertex 
$c_{i}$ and $f_{i}(x)\in K[x]$ with a non-zero constant term which, without
loss of generality, we may assume to be $1$. Then $v_{i}$ will be the
non-zero constant term of $f_{i}(c_{i})$. It is clear that, for a fixed $%
i\in X$, $H_{i}=\{u\in E^{0}:u\ngeq v_{i}\}$ is a hereditary saturated set
and $(E\backslash(H_{i},B_{H_{i}}))^{0}=E^{0}\backslash H_{i}$ is downward
directed. Then, for two distinct non-conjugate irreducible polynomials $%
p(x),q(x)\in K[x,x^{-1}]$, we have $P=I(H_{i},B_{H_{i}})+<p(c_{i})>$ and $%
Q=I(H_{i},B_{H_{i}})+<q(c_{i})>$ are prime ideals of $L$ such that neither
contains the other. To see this, note that in $\bar{L}=L/I(H_{i},B_{H_{i}})$%
, $M=<c_{i}^{0}>$ contains both 
\begin{equation*}
\bar{P}=P/(H_{i},B_{H_{i}})=<p(c_{i})>\text{ and }\bar{Q}%
=Q/(H_{i},B_{H_{i}})=<q(c_{i})>.
\end{equation*}
By Proposition \ref{M=<c^0>}, $M\cong M_{\Lambda}(K[x,x^{-1}])$ and that the
ideal lattices of $M_{\Lambda}(K[x,x^{-1}])$ and $K[x,x^{-1}]$ are
isomorphic. Consequently, $\bar{P}$ and $\bar{Q}$ are maximal ideals of $M$
and hence neither contains the other. This contradiction shows that all the
ideals of $L$ must be graded. This proves (2).

Now (2) implies (4), because, if $A,B,I$ are ideals of $L$ such that $%
AB\subseteq I$. $\ $Since ideals of $L$ are all graded, $AB=A\cap B$ (Lemma %
\ref{Property of graded ideal}). Thus $A\cap B\subseteq I$. Since $A\cap B=A$
or $B$, we have $A\subseteq I$ or $B\subseteq I$.

Also (4) implies (1), since a prime ideal is always strongly irreducible.

Finally, (2) $\Leftrightarrow $ (3), since, by (Proposition 2.9.9, \cite{AAS}%
), the graph $E$ satisfies Condition (K) if and only if every ideal of $L$
is graded and, by (Theorem 2.5.8, \cite{AAS}), the map $I(H,S)\mapsto (H,S)$
is an order preserving bijection between graded ideals of $L$ and admissible
pairs $(H,S)$.
\end{proof}

\begin{example}
\label{Ideals form a chain} Suppose $E$ is the following graph

\bigskip

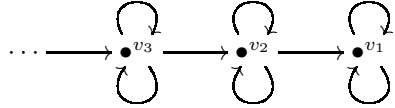
\begin{equation*}
\xymatrix{ \cdots \ar[r] & \bullet^{v_3} \ar@(ul,ur) \ar@(dr,dl) \ar[r]  & \bullet^{v_2} \ar@(ul,ur) \ar@(dr,dl) \ar[r] &
\bullet^{v_1} \ar@(ul,ur) \ar@(dr,dl)}
\end{equation*}

\bigskip

\noindent Then clearly $E$ satisfies Condition (K) and is
row-finite. In particular, $B_{H}$ is empty for every hereditary saturated
subset of vertices $H\subseteq E^{0}$. Moreover, the non-empty proper
hereditary saturated subsets of vertices in $E$ are precisely the sets of
the form $\{v_{i}:i\geq n\}$, where $n$ is a positive integer and these form
an ascending chain under set inclusion. It follows that the admissible pairs 
$(H,S)$ form a chain under the partial order of the admissible pairs as
defined in Section 2. Thus the graph $E$ satisfies Condition (3) of Theorem %
\ref{Every Id is str. Irred}.
\end{example}

\section{Strongly Prime Ideals of Leavitt path algebras}

\noindent As noted earlier, a prime ideal $P$ containing the intersection ${%
\displaystyle\bigcap\limits_{i=1}^{n}}A_{i}$ of finitely many ideals\ $A_{i}$
will contain one of the ideals $A_{i}$. But, for the intersection of
infinitely many ideals, the corresponding statement does not hold. For
example, in the ring $\mathbb{Z}$ of integers, the zero ideal $\{0\}={%
\displaystyle\bigcap\limits_{n=1}^{\infty}}2^{n}\mathbb{Z}$, but $2^{n}%
\mathbb{Z}\neq\{0\}$. In \cite{JOT}, Jayaram, Oral and Tekir study the
ideals of a commutative ring having the desired property for infinite
intersections and call them strongly prime ideals.

\begin{definition}
(\cite{JOT}) An ideal $P$ of a ring $R$ is called a strongly prime ideal if $%
P\supseteq{\displaystyle\bigcap\limits_{i\in X}}A_{i}$, where $X$ is an
arbitrary index set and the $A_{i}$ are ideals of $R$ implies that $%
P\supseteq A_{i}$ for some $i\in X$. A ring $R$ is called strongly
zero-dimensional, if every prime ideal of $R$ is strongly prime.
\end{definition}

In this section we characterize the strongly prime ideals of a Leavitt path
algebra $L$. We give necessary and sufficient conditions under which a given
ideal $I$ of $L$ can be factored as a product of strongly prime ideals. We
describe both algebraically and graphically when every ideal of $L$ admits a
factorization as a product of strongly prime ideals. Finally, Leavitt path
algebras which are strongly zero-dimensional are fully characterized.

\begin{remark}
\textrm{In their definition of a strongly prime ideal $P$ in \cite{JOT}, the
authors assume to start with that the ideal $P$ is a prime ideal satisfying
the stated property. As is clear from our definition above, we do not assume
a priori that $P$ is a prime ideal. We will show in Theorem \ref{Strongly
primes of LPAs} that, for Leavitt path algebras, such an ideal $P$ is always
a prime ideal. }
\end{remark}

Clearly a prime (and hence a strongly prime) ideal of a ring is strongly
irreducible. But a strongly irreducible ideal need not be strongly prime.
For instance, in the ring $\mathbb{Z}$ of integers, $2\mathbb{Z}$ is
strongly irreducible (and also a prime ideal), but $2\mathbb{Z}$ is not
strongly prime since ${\displaystyle\bigcap\limits_{n=1}^{\infty}}3^{n}%
\mathbb{Z}=0\in$ $2\mathbb{Z}$, but $3^{n}\mathbb{Z}\nsubseteq2\mathbb{Z}$
for any $n\geq1$.

We begin with some preliminary results.

\begin{lemma}
\label{No PID} No ideal in a principal ideal domain $R$ is strongly prime.
\end{lemma}

\begin{proof}
Let $I=<a>$ be a non-zero ideal of $R$. Choose a prime (equivalently;
irreducible) element $p$ in $R$ such that $p\nmid a$. Now ${\displaystyle%
\bigcap\limits_{n=1}^{\infty}} <p^{n}>=\{0\}\subseteq<a>$, but $%
<p^{n}>\nsubseteq<a>$ for any $n$, since otherwise, $p^{n}=ab$ for some $%
b\in R$ and by the uniqueness of prime factorization, $a=p^{k}$ for some $k$%
, a contradiction. Also the zero ideal $<0>$ of $R$ cannot be strongly
prime, since for any prime element $p$ in $R$, ${\displaystyle%
\bigcap\limits_{n=1}^{\infty}} <p^{n}>=<0>$.
\end{proof}

\begin{corollary}
\label{No K[x,x^-1]} No ideal in the ring $M_{\Lambda}(K[x,x^{-1}])$ is
strongly prime.
\end{corollary}

\begin{proof}
By Proposition 3.3, the ideal lattices of $M_{\Lambda}(K[x,x^{-1}])$ and the
principal ideal domain $K[x,x^{-1}]$ are isomorphic. Then the result follows
from Lemma \ref{No PID}.
\end{proof}

The next Proposition states that a strongly prime ideal of $L=L_{K}(E)$ must
be a graded ideal.

\begin{proposition}
\label{Non-graded prime power not} If $P$ is a non-graded prime ideal of a
Leavitt path algebra $L_{K}(E)$, then, for any $n\geq1$, $P^{n}$ is not a
strongly prime ideal.
\end{proposition}

\begin{proof}
By Theorem 2.2, $P=I(H,B_{H})+<p(c)>$ where $H=P\cap E^{0}$, $E\backslash
(H,B_{H})$ is downward directed, $c$ is a cycle without exits in $%
E\backslash(H,B_{H})$ and $p(x)$ is an irreducible polynomial in $%
K[x,x^{-1}] $. Then $P^{n}=I(H,B_{H})+<p^{n}(c)>$. In $\overline{L_{K}(E)}%
=L_{K}(E)/I(H,B_{H})\cong L_{K}(E\backslash(H,B_{H}))$, 
\begin{equation*}
\overline{P}^{n}=P^{n}/I(H,B_{H})=<p^{n}(c)>\subseteq<c^{0}>=M.
\end{equation*}
Now, by Proposition \ref{M=<c^0>}, $M\cong M_{\Lambda}(K[x,x^{-1}])$ for a
suitable index set $\Lambda$. Then, by Corollary \ref{No K[x,x^-1]}, ($%
\overline{P}^{n}$ and hence) $P^{n}$ is not a strongly prime ideal in $%
L_{K}(E)$.
\end{proof}

We now proceed to give a complete description of the strongly prime ideals
in a Leavitt path algebra $L_{K}(E)$. In particular, it shows that a
strongly prime ideal must be a prime ideal. In its proof, we shall be using
the following definition.

\begin{definition}
\label{CSP} (\cite{ABR}) Given a graph $E$, we say that $E^{0}$ satisfies
the \textit{countable separation property} (for short, CSP), if there is a
non-empty countable subset $S$ of $E^{0}$ such that for every $u\in E^{0}$
there is a $v\in S$ such that $u\geq v$.
\end{definition}

The CSP condition turns out to be essential in the description of primitive
ideals of Leavitt path algebras as noted in the next theorem. Recall that a
ring $R$ is called a right primitive ring if there exists a faithful simple
right $R$-module. An ideal $I$ of $R$ is called a right primitive ideal if $%
R/I$ is a right primitive ring.

\begin{theorem}
\label{Primitive ideal in an LPA} (Theorem 4.3, \cite{R-1}) Let $E$ be an
arbitrary graph and let $P$ be an ideal of $L_{K}(E)$ with $P\cap E^{0}=H$.
Then $P$ is a primitive ideal if and only if $P$ satisfies one of the
following:

\begin{enumerate}
\item $P$ is a non-graded prime ideal;

\item $P$ is a graded prime ideal of the form $I(H,B_{H}\backslash\{u\});$

\item $P$ is a graded ideal of the form $I(H,B_{H})$ (where $B_{H}$ may be
empty) and $(E\backslash(H,B_{H}))^{0}=E^{0}\backslash H$ is downward
directed and satisfies Condition (L) and the CSP.
\end{enumerate}
\end{theorem}

\begin{definition}
\label{strong CSP} (i) We say $E^{0}$ satisfies the \textit{strong CSP}, if $%
E^{0}$ satisfies the CSP such that the corresponding non-empty countable set 
$S$ of vertices is contained in every non-empty hereditary saturated subset
of $E^{0}$.

(ii) A primitive ideal $P$ of $L_{K}(E)$ with $gr(P)=I(H,S)$ is called 
\textit{strongly primitive} if $(E\backslash(H,S))^{0}$ satisfies the strong
CSP.

(iii) We say $L_{K}(E)$ is strongly primitive, if $L_{K}(E)$ is a primitive
ring such that $E^{0}$ satisfies the strong CSP.
\end{definition}

\begin{example}
\textrm{Any non-graded prime ideal $P$ of a Leavitt path algebra $L$ is
strongly primitive. Because, by Theorem \ref{Prime ideals}, $%
P=I(H,B_{H})+<p(c)>$ where $c$ is a cycle without exits in the downward
directed set $(E\backslash(H,B_H))^{0}$ based at a vertex $v$. Then $%
(E\backslash (H,B_{H}))^{0}$ satisfies the strong CSP with respect to $\{v\}$%
, thus making $I$ strongly primitive. Also any graded prime ideal of the
form $I(H,B_{H}\backslash\{u\})$ is strongly primitive since it is primitive
(\cite{R-1}Theorem 4.3) and $(E\backslash(H,B_{H}\backslash\{u\}))^{0}$
satisfies the strong CSP with respect to $\{u^{\prime}\}$. }
\end{example}

\begin{lemma}
Every strongly prime ideal of a Leavitt path algebra is graded prime.
\end{lemma}

\begin{proof}
Let $I$ be a strongly prime ideal of a Leavitt path algebra $L$. In
particular, $I$ is irreducible and so, by Lemma \ref{Irreducible Ideals of L}%
, $I=P^{n}$, a power of a prime ideal $P$. By Proposition \ref{Non-graded
prime power not}, $P$ must be a graded ideal and so $I=P^{n}=P$ is a graded
prime ideal.
\end{proof}

\begin{theorem}
\label{Strongly primes of LPAs}The following properties are equivalent for
an ideal $I$ of $L=L_{K}(E)$ with $I\cap E^{0}=H$:

\begin{enumerate}
\item $I$ is a strongly prime ideal;

\item $I$ is a graded and strongly primitive ideal;

\item $I=I(H,S)$ is graded and $E\backslash(H,S)$ is downward directed
satisfying Condition (L) and the strong CSP.
\end{enumerate}
\end{theorem}

\begin{proof}
Assume (1). By above lemma $I$ is a graded prime ideal, say, $I=I(H,S)$
where $H=I\cap E^{0}$and $(E\backslash(H,S))^0$ is downward directed. We
claim that $I$ must be a primitive ideal of $L$. Because, otherwise, $\{0\}$
is not a primitive ideal of $L/I$, so the primitive ideals of $L/I$ are
non-zero and their intersection is $\{0\}$ since the Jacobson radical of $%
L/I $ is $\{0\}$ due the fact that $L/I\cong L_{K}(E\backslash(H,S))$. This
means that $I$ is the intersection of all the primitive ideals properly
containing $I$, contradicting the fact that $I$ is strongly prime. Thus $I$
must be a (graded) primitive ideal and, by Theorem \ref{Primitive ideal in
an LPA}, $(E\backslash(H,S))^0$ is downward directed, satisfies Condition
(L) and has the CSP with respect to a non-empty countable subset $C$ of
vertices. Clearly, $\{0\}$ is strongly prime in $L/I\cong
L_{K}(E\backslash(H,S))$. Let $X=\cap_{j\in J}H_{j}$ be the intersection of
all non-empty hereditary saturated subsets $H_{j}$ of vertices in $%
E\backslash(H,S)$. Now $X$ is not empty since, otherwise, $\{0\}=\cap _{j\in
J}<H_{j}>$ in $L/I$ and this would contradict the fact that $\{0\}$ is
strongly prime. We claim that, for each vertex $v\in E\backslash(H,S)$,
there is a vertex $w\in X\cap C$ such that $v\geq w$. To see this, let $u$
be a fixed vertex in $X$. By downward directness, there is a vertex $%
w^{\prime}\in E\backslash(H,S)$ such that $u\geq w^{\prime}$ and $v\geq
w^{\prime}$. By CSP, there is a $w\in C$ such that $w^{\prime}\geq w$. Since 
$X$ is hereditary, $u\geq w$ implies that $w\in X$. Thus $v\geq w\in X\cap C$
and we conclude that $E\backslash(H,S)$ satisfies the strong CSP with
respect to $X\cap C$. Consequently, $I$ is a graded and strongly primitive
ideal. This proves (2).

Assume (2). Let $I$ be a graded and strongly primitive ideal of the form $%
I(H,S)$ with $E\backslash(H,S)$ satisfying Condition (L) and the strong CSP
with respect to a non-empty countable set $C$ of vertices. To show that $I$
is strongly prime, suppose $\{A_{j}:j\in J\}$ is an arbitrary family of
ideals of $L$ such that $A_{j}\nsubseteq I$ for all $j\in J$. Consider $%
\bar {L}=L/I(H,S)\cong L_{K}(E\backslash(H,S))$. Let, for each $j$, $\bar{A}%
_{j}=(A_{j}+I(H,S))/I(H,S)$ . Identifying $L/I(H,S)$ with $%
L_{K}(E\backslash(H,S))$ under this isomorphism, observe that the non-zero
ideal $\bar{A}_{j}$ must contain a vertex. Otherwise, as $%
(E\backslash(H,S))^{0}$ is downward directed, Lemma 3.5 in \cite{R-1} will
imply that $\bar{A}_{j}=<f(c)>$ where $f(x)\in K[x]$ is a polynomial with a
non-zero constant term and $c$ is a cycle without exits in $E\backslash(H,S)$%
. This is a contradiction, since $E\backslash(H,S)$ satisfies Condition (L).
Thus $H_{j}=\bar{A}_{j}\cap(E\backslash(H,S))^{0}\neq\emptyset$ for all $%
j\in J$. By the strong CSP of $(E\backslash(H,S))^{0}$, the countable set $%
C\subseteq H_{j}$ for all $j\in J$. Consequently, $\cap_{j\in J}\bar{A}%
_{j}\supseteq<C>\neq0$. This means that $\cap_{j\in J}A_{j}\nsubseteq I(H,S)$%
. Hence $I$ is strongly prime, thus proving (1).

The equivalence of (2) and (3) follows from the definition of strongly
primitive ideal.
\end{proof}

\begin{remark}
\textrm{We shall call a ring $R$ a strongly prime ring if $\{0\}$ is a
strongly prime ideal. It is then clear from Theorem \ref{Strongly primes of
LPAs} that a Leavitt path algebra $L_{K}(E)$ is strongly prime if and only
if $E^{0}$ is downward directed satisfying the strong CSP and Condition (L)
if and only if $L_{K}(E)$ is a strongly primitive ring. }
\end{remark}

The next result describes conditions under which an ideal $I$ of a Leavitt
path algebra $L_{K}(E)$ can be factored as a product of strongly prime
ideals.

\begin{theorem}
\label{Product of str. comp. irred. ideals} The following properties are
equivalent for an ideal $I$ of a Leavitt path algebra $L=L_{K}(E)$:

\begin{enumerate}
\item $I=P_{1}\cdots P_{n}$ is a product of strongly prime ideals $P_{j}$;

\item $I=P_{1}\cap\cdots\cap P_{n}$ is an intersection of finitely many
strongly prime ideals $P_{j}$;

\item $I$ is a graded ideal, say $I=I(H,S)$ \ where $H=I\cap E^{0}$ and $%
S\subseteq B_{H}$ such that $E\backslash(H,S)^{0}$ is an irredundant union
of finitely many maximal tails $M_{1},...,M_{t}$ with each subset $M_{j}$
satisfying Condition (L) and the strong CSP with respect to a countable
subset $C_{j}\subseteq M_{j}$.
\end{enumerate}
\end{theorem}

\begin{proof}
To prove the equivalence of (1) and (2), note that each $P_{j}$ is a graded
ideal, by Theorem \ref{Strongly primes of LPAs}. Then, by Lemma \ref%
{Property of graded ideal}, $P_{1}\cdots P_{n}=P_{1}\cap \cdots\cap P_{n}$.

Assume (2). If necessary remove some of the ideal $P_{j}$ and assume that $%
I=P_{1}\cap\cdots\cap P_{t}$ is an irredundant intersection of strongly
prime ideals. By Theorem \ref{Strongly primes of LPAs}, each ideal $P_{j}$
is graded, say, $P_{j}=I(H_{j},S_{j})$ such that $E\backslash (H_{j},S_{j})$
is downward directed and satisfies both Condition (L) and the strong CSP
with respect to a countable subset of vertices $C_{j}$. So $%
I=P_{1}\cap\cdots\cap P_{t}$ is graded, say $I=I(H,S)$ where $H=I\cap E^{0}$
and $S\subseteq B_{H}$. In $L/I\cong L_{K}(E\backslash(H,S)$, $\{0\}=\bar{P}%
_{1}\cap\cdots\cap\bar{P}_{t}$ is an irredundant intersection, where $\bar{P}%
_{j}=P_{j}/I(H,S)$ is strongly prime and hence a prime ideal for all $%
j=1,\ldots, t$. Let $H_{j}=\bar{P}_{j}\cap E\backslash(H,S)^{0}$. Then $%
M_{j}=[E\backslash(H,S)^{0}]\backslash H_{j}$ is a maximal tail. As $\bigcap
\limits_{j=1}^{t}H_{j}=\emptyset$, we have $E\backslash(H,S)^{0}=\bigcup
\limits_{j=1}^{t}M_{j}$ is an irredundant union of the maximal tails $M_{j}$%
. Since each $P_{j}/I(H,S)$ is strongly prime, each subset $M_{j}$ satisfies
Condition (L) and the strong CSP with respect to a countable subset $%
C_{j}\subseteq M_{j}$. This proves (3)

Assume (3). In $L_{K}(E\backslash(H,S))$, $H_{j}=(E\backslash(H,S))^{0}%
\backslash M_{j}$ is a hereditary saturated set for each $j=1,\ldots, t$ and
let $\bar{P}_{j}=I(H_{j},B_{H_{j}})$ Since ${\displaystyle\bigcap
\limits_{j=1}^{t}}H_{j}=\emptyset$, ${\displaystyle\bigcap\limits_{j=1}^{t}}%
\bar{P}_{j}=\{0\}$. Identifying $L/I$ with $L_{K}(E\backslash(H,S))$, each $%
\bar{P}_{j}=P_{j}/I$ for some ideal $P_{j}$ of $L$ containing $I$. Using
Theorem \ref{Strongly primes of LPAs} and the hypothesis on $M_{j}$, we
conclude that each $P_{j}$ is strongly prime and that $I=P_{1}\cap\cdots\cap
P_{t}$. This proves (2).
\end{proof}

\bigskip

\noindent To illustrate the above, we consider the following simple example.

\begin{example}
\textrm{Let $E$ be the following graph }

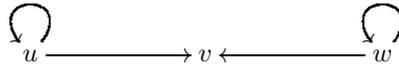
\textrm{\bigskip}

\textrm{%
\begin{equation*}
\xymatrix{ && u\ar@{->}[rr]\ar@(ur,ul) && v\ar@{<-}[rr] && w\ar@(ur,ul)\\ }
\end{equation*}
}

\textrm{\bigskip}

\textrm{\noindent Clearly, $P=<u>=<\{u,v\}>$ and $Q=<w>=<\{v,w\}>$ are
strongly prime ideals and $I=<v>=PQ=P\cap Q$.}
\end{example}

\bigskip

\begin{remark}
\label{Uniques prod. str. primes}Just as we did in Section 3, Theorem \ref%
{Uniqueness}, a natural question is to inquire about the uniqueness of
factorization of an ideal $I$ of $L$ as an irredundant product or an
irredundant intersection of finitely strongly prime ideals. Since a strongly
prime ideal is, in particular, a prime ideal, the uniqueness of such a
factorization follows from the uniqueness of factorizing an ideal $I$ as an
irredundant product/intersection of prime ideals (see e.g. Propositions 2.6
and 3.4 in \cite{EMR}).
\end{remark}

\bigskip

\noindent In the next theorem, we characterize when every ideal of a Leavitt
path algebra is a product of strongly prime ideals.

\begin{theorem}
\label{Everyone product of str.comp.irreducibles} The following properties
are equivalent for a Leavitt path algebra $L_{K}(E)$:

\begin{enumerate}
\item Every ideal of $L_{K}(E)$ is a product of strongly prime ideals;

\item Every ideal of $L_{K}(E)$ is graded, $L$ is a generalized ZPI ring,
every homomorphic image of $L$ is a Leavitt path algebra and is either
strongly prime or contains only a finite number of minimal prime ideals each
of which is strongly primitive;

\item The graph $E$ satisfies Condition (K) and for every quotient graph $%
E\backslash(H,S)$, $E\backslash(H,S)^{0}$is either downward directed
satisfying the strong CSP or is the union of finitely many maximal tails $%
S_{j}$, each of which satisfies the strong CSP with respect to a countable
subset $C_{j}\subseteq S_{j}$.
\end{enumerate}
\end{theorem}

\begin{proof}
Assume (1). First note that every ideal of $L$ is, in particular, a product
of prime ideals and so $L$ is a generalized ZPI ring. If $I$ is an arbitrary
ideal, then $I=P_{1}\cdots P_{n}$ is a product of strongly prime ideals $%
P_{j}$ implies $I=P_{1}\cap\cdots\cap P_{n}$ as the $P_{j}$ are all graded.
Thus $I$ is a graded ideal. Removing appropriate factors $P_{j}$ and
re-indexing, we may assume that $I=P_{1}\cap\cdots\cap P_{m}$ is an
irredundant intersection of graded strongly primitive ideals. If $I$ is a
prime ideal, then $I=P_{j}$ for some $j$ and so $I$ is a strongly prime
ideal of $L$. Suppose $I$ is not a prime ideal. In $\bar{L}=L/I$, $\{0\}=%
\bar{P}_{1}\cap\cdots\cap\bar{P}_{m}$ is an irredundant intersection, where $%
\bar{P}_{j}=P_{j}/I$. Since $L/I\cong L_{K}(E\backslash(H,S))$, Theorem 6.5
of \cite{R-2} and its proof implies that $L/I$ contains finitely many
minimal prime ideals $Q_{1},\ldots ,Q_{k}$ \ which are all graded and that
we have an irredundant intersection $Q_{1}\cap\cdots\cap Q_{k}=\{0\}$. By
the minimality of the prime ideals $Q_{j}$ and by irredundancy of the two
intersections, we obtain that $k=m$ and $\{Q_{1},\ldots, Q_{m}\}=\{\bar{P}%
_{1},\ldots,\bar{P}_{m}\}$ Thus $L/I$ contains finitely many minimal prime
ideals each of which is strongly primitive. This proves (2).

Assume (2). Since every ideal is graded, the graph $E$ satisfies Condition
(K) (Proposition 2.9.9, \cite{AAS}). For a given graded ideal $I(H,S)$, we
are given that $L_{K}(E\backslash(H,S))\cong L/I(H,S)$ is strongly prime or\
contains only a finite number of minimal prime ideals. In the former case,
since Condition (K) implies Condition (L), we obtain from Theorem \ref%
{Strongly primes of LPAs} that $E\backslash(H,S)$ is downward directed and
satisfies the strong CSP. On the other hand, suppose $L_{K}(E\backslash
(H,S))\cong L/I(H,S)$ contains only a finite number of minimal prime ideals $%
P_{1},\ldots, P_{k}$ all of which are graded and strongly prime. For each $%
j=1,\ldots, k$, write as $P_{j}=I(H_{j},S_{j})$. Since every non-zero prime
ideal of $L_{K}(E\backslash(H,S))$ contains one of these $P_{j}$ and since
the intersection of all the prime ideals of the Leavitt path algebra $%
L_{K}(E\backslash(H,S))$ is zero,\ we conclude that $%
\cap_{j=1}^{m}I(H_{j},S_{j})=\{0\}$. As $I(H_{j},S_{j})$ are graded strongly
primitive, $M_{j}=[E\backslash(H,S)\backslash(H_{j},S_{j})]^{0}$ is a
maximal tail satisfying the strong CSP with respect to a countable subset $%
C_{j}\subseteq S_{j}$ and $E\backslash(H,S)^{0}=\cup_{j=1}^{m}M_{j}$. This
proves (3).

Assume (3). Let $I$ be an arbitrary ideal of $L$. Condition (K) implies that 
$I$ is graded, say, $I=I(H,S)$ and that $E\backslash(H,S)$ satisfies
Condition (L). By hypothesis, $E\backslash(H,S)^{0}=\cup_{j=1}^{m}S_{j}$
where each $S_{j}$ is a maximal tail satisfying the strong CSP (and
Condition (L)). Then, for each $j=1,\cdot\cdot\cdot,m$, $H_{j}=E%
\backslash(H,S)^{0}\backslash S_{j}$ will be a hereditary saturated subset
of $E\backslash(H,S)^{0}$ and, by Theorem \ref{Strongly primes of LPAs}, $%
Q_{j}=I(H_{j},B_{H_{j}})$ will be a graded strongly prime ideal of $%
L_{K}(E\backslash(H,S))$. Clearly $\cap _{j=1}^{m}I(H_{j},B_{H_{j}})=\{0\}$.
Now $L/I(H,S)\cong L_{K}(E\backslash (H,S))$ and identify these two rings
under this isomorphism. For each $j=1,\ldots, m$, let $P_{j}$ be the ideal
of $L$ such that \ $P_{j}/I(H,S)=Q_{j}$. Then $P_{j}$ is a strongly prime
ideal and $I=\cap_{j=1}^{m}P_{j}$. As $I$ is graded, we appeal to Lemma \ref%
{Property of graded ideal} to conclude that $I=P_{1}\cdots P_{m}$, a product
of strongly prime ideals. This proves (1).
\end{proof}

For an illustration for the above theorem, let us recall the Example \ref%
{row-finite/countable graph} from previous section.

\begin{example} \rm
Let $E$ be the following graph 

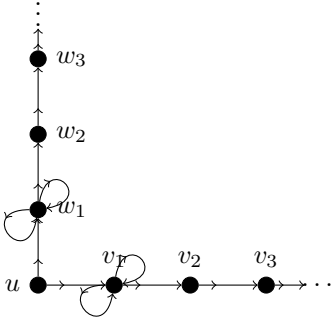
\begin{tikzpicture}[x=1cm, y=1cm,every edge/.style={draw, postaction={decorate,decoration={markings,mark=at position .3 with {\arrow{>}}}}}]
\node[vertex, label = {left: $u$}] (u) at (-3, 1) {}; 
\node[vertex, label = {above: $v_1$}] (v1) at (-2,1) {};
\node[vertex, label = {above: $v_2$}] (v2) at (-1, 1) {}; 
\node[vertex, label = {above: $v_3$}] (v3) at (0, 1) {};  
\node at (.7,1) {\dots};  
\node[vertex, label = {right: $w_1$}] (w1) at (-3,2) {};
\node[vertex, label = {right: $w_2$}] (w2) at (-3, 3) {}; 
\node[vertex, label = {right: $w_3$}] (w3) at (-3, 4) {}; 
\node at (-3,4.7) {\vdots};   
\path[->] (u)  edge (v1);
\path[->] (v1)  edge (v2);
\path[->] (v2)  edge  (v3);
\path[->] (v3)  edge (0.5,1);  
\path[->] (u)  edge (w1);
\path[->] (w1)  edge (w2);
\path[->] (w2)  edge  (w3);
\path[->] (w3)  edge (-3,4.4);

\path [->] (w1) edge [in=10 ,out=80, looseness=13,->](w1);
\path [->] (w1) edge [in=260 ,out=180,looseness=13,->](w1);
\path [->] (v1) edge [in=10 ,out=80, looseness=13,->](v1);
\path [->] (v1) edge [in=260 ,out=180,looseness=13,->](v1);

\end{tikzpicture} 

\noindent The graph is row-finite and satisfies the Condition (K).
The proper  hereditary saturated subsets of this graph are precisely the
empty set $\emptyset $ and the sets $A_{1}=\{v_{n}:n\geq 1\}$; $%
A_{2}=\{v_{n}:n\geq 2\}$; $B_{1}=\{w_{n}:n\geq 1\}$; $B_{2}=\{w_{n}:n\geq 2\}
$. For each $i=1,2$, $E^{0}\setminus A_{i}$ and also $E^{0}\setminus B_{i}$
is the union of \ at most two maximal tails\ each of which satisfy strong
CSP with respect to itself. Since $A_{1}\cap A_{2}=\emptyset ,$ $%
E^{0}\setminus \emptyset =E^{0}\setminus (A_{1}\cap A_{2})=(E^{0}\setminus
A_{1})\cup (\mathrm{E^{0}\setminus A_{2}})$ is also the union of two maximal
tails satisfying the strong CSP with themselves. Thus the graph satisfies
Condition (3) of Theorem \ref{Everyone product of str.comp.irreducibles}. 
\end{example}

\bigskip

\noindent We proceed to characterize Leavitt path algebras in which each
ideal is strongly prime.

\begin{theorem}
\label{sp} The following are equivalent for any Leavitt path algebra $%
L=L_{K}(E)$:

\begin{enumerate}
\item Every ideal of $L$ is strongly prime;

\item All the ideals of $L$ are graded strongly primitive and form a chain
under set inclusion;

\item The graph $E$ satisfies Condition (K), the admissible pairs $(H,S)$
form a chain under the partial ordering of the admissible pairs and for each
admissible pair $(H,S)$, $E\backslash(H,S)$ is downward directed and
satisfies the strong CSP.
\end{enumerate}
\end{theorem}

\begin{proof}
Assume (1). Let $I$ be any ideal of $L$. Since $I$ is strongly prime, it is
graded and strongly primitive, by Theorem \ref{Strongly primes of LPAs}. To
show that the ideals form a chain, observe that every ideal of $L$ is prime,
as a strongly prime ideal is always prime. Then, for any two ideals $A,B$, $%
A\cap B$ is a prime ideal and so $AB\subseteq A\cap B$ implies that $%
A\subseteq A\cap B$ or $B\subseteq A\cap B$ and this implies that either $%
A\subseteq B$ or $B\subseteq A$. Thus the ideals of $L$ form a chain under
inclusion. This proves (2).

Assume (2). Since every ideal of $L$ is graded, the graph $E$ satisfies
Condition (K), by Proposition 2.9.9 of \cite{AAS}. Now every ideal of $L$ is
graded and so is of the form $I(H,S)$ for some admissible pair $(H,S)$.
Since the ideals of $L$ form a chain, it is clear from the order preserving
bijection between the graded ideals of $L$ and the admissible pairs that the
admissible pairs form a chain under their partial ordering. Now, for a given
admissible pair $(H,S)$, the corresponding ideal $I(H,S)$ is strongly prime
and so, by Theorem \ref{Strongly primes of LPAs}, $E\backslash(H,S)$ is
downward directed and satisfies the strong CSP. This proves (3).

Assume (3). Let $I$ be an arbitrary ideal of $L$. Since $E$ satisfies
Condition (K), $I$ is a graded ideal (Proposition 2.9.9, \cite{AAS}), say $%
I=I(H,S)$ where $H=I\cap E^{0}$. Now $E\backslash(H,S)$ satisfies Condition
(K) and hence Condition (L). Since, by supposition, it is downward directed
and satisfies the strong CSP, we appeal to Theorem \ref{Strongly primes of
LPAs} to conclude that $I$ is a strongly prime ideal, thus proving (1).
\end{proof}

\begin{example}
\textrm{Consider the graph $E$ below. 
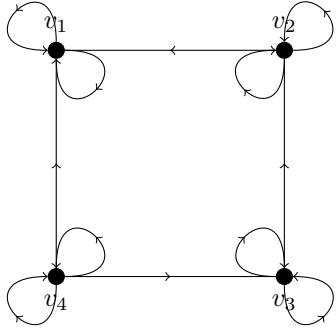
\begin{equation*}
\begin{tikzpicture}[x=3cm, y=3cm,every edge/.style={draw,
postaction={decorate,decoration={markings,mark=at position 0.5 with
{\arrow{>}}}}}] \node[vertex] (v1) at (0,1) [label=above:$v_{1}$] {};
\node[vertex] (v2) at (1,1) [label=above:$v_{2}$] {}; \node[vertex] (v3) at
(1,0) [label=below:$v_{3}$] {}; \node[vertex] (v4) at (0,0)
[label=below:$v_{4}$] {}; \path (v2) edge (v1) (v3) edge (v2) (v4) edge (v3)
(v4) edge (v1) (v1) edge [in=180,out=90 ,looseness=20,->] (v1) (v2) edge
[in=90,out=0,looseness=20,->] (v2) (v3) edge [in=0,out=270,looseness=20,->]
(v3) (v4) edge [in=90,out=0,looseness=20,->] (v4) (v1) edge
[in=270,out=360,looseness=20,->] (v1) (v2) edge
[in=180,out=270,looseness=20,->] (v2) (v3) edge [in=90
,out=180,looseness=20,->](v3) (v4) edge [in=180,out=270,looseness=20,->](v4)
; \end{tikzpicture}
\end{equation*}
The proper hereditary saturated subset of vertices are $H_{0}=\emptyset
,H_{1}=\{v_{1}\},H_{2}=\{v_{1},v_{2}\}$ and $H_{3}=\{v_{1},v_{2},v_{3}\}$.
It is easy to see that, $E^{0}\backslash H_{0}=E^{0},E^{0}\backslash
H_{1},E^{0}\backslash H_{2},E^{0}\backslash H_{3}$ are all downward directed
and satisfy the strong CSP with respect to $\{v_{1}\},\{v_{2}\},\{v_{3}\},%
\{v_{4}\}$, respectively. Hence every proper ideal of $L_{K}(E)$ is strongly
prime. The graph $E$ is finite and satisfies Condition (K), thus
illustrating Theorem \ref{sp}. }
\end{example}

The next theorem describes when a Leavitt path algebra is strongly
zero-dimensional, that is, when every prime ideal of $L$ is strongly prime.
In its proof, we shall be using the following concept of an extreme cycle.

\begin{definition}
\cite{AAS} A cycle $c$ in a graph $E$ is said to an extreme cycle, if it has
exits and for every path $\alpha$ with $s(\alpha)\in c^{0}$, there is a path 
$\beta$ with $s(\beta)=r(\alpha)$ such that $r(\beta)\in c^{0}$.
\end{definition}

\begin{theorem}
\label{psp}

\begin{enumerate}
\item Suppose $E$ is a finite graph (or more generally suppose $E^{0}$ is
finite). Then every prime ideal of $L_{K}(E)$ is strongly prime if and only
if the graph $E$ satisfies Condition (K);

\item Let $E$ be an arbitrary graph. Then every prime ideal of $L_{K}(E)$ is
strongly prime if and only if the graph $E$ satisfies Condition (K) and
every quotient graph $E\backslash(H,S)$ which is downward directed satisfies
the strong CSP.
\end{enumerate}
\end{theorem}

\begin{proof}
Suppose every prime ideal of $L_{K}(E)$ is strongly prime, then every prime
ideal of $L_{K}(E)$ is graded, since a strongly prime ideal always graded
(Theorem \ref{Strongly primes of LPAs}). This implies, by Corollary 3.13 of 
\cite{R-1}, that the graph $E$ satisfies Condition (K) which, by Proposition
2.9.9 of \cite{AAS}, is equivalent to every ideal of $L_{K}(E)$ being a
graded ideal.

(1). Assume now that $E^{0}$ is finite and that $E$ satisfies Condition (K).
Let $P$ be any prime ideal of $L_{K}(E)$ which, being graded, will be of the
form $P=I(H,S)$ where $H=P\cap E^{0}$. Since $(E\backslash(H,S))^{0}$ is
finite, Lemma 3.7.10 of \cite{AAS} implies that every path in $E\backslash
(H,S)$ ends at a sink, a cycle without exits or an extreme cycle. If there
is a sink in $(E\backslash(H,S))^{0}$, then by downward directness, there
can be only one sink, say $w$ in $(E\backslash(H,S))^{0}$and, moreover,
every non-empty hereditary subset of the downward directed set $(E\backslash
(H,S))^{0}$ will contain $w$ and hence $(E\backslash(H,S))^{0}$ will satisfy
the strong CSP with respect to $\{w\}$. If there are no sinks in $%
(E\backslash(H,S))^{0}$, then, since $E$ satisfies Condition (K), every
cycle will have exits and so every path in $E\backslash(H,S)$ will end at an
extreme cycle. Since $(E\backslash(H,S))^{0}$ is finite, there are only
finitely many extreme cycles, say $c_{1},\ldots,c_{n}$. Fix a vertex $v\in
c_{1}^{0}$. We claim that every vertex $u\in(E\backslash(H,S))^{0}$
satisfies $u\geq v$. To see this, note that, since every path ends at one of
the cycles $c_{1},\ldots, c_{n}$, $u\geq w\in c_{j}^{0}$ for some $j$. Since 
$(E\backslash(H,S))^{0}$ is downward directed, there is a vertex $v_{0}$
such that $v\geq v_{0}$ and $w\geq v_{0}$. Suppose $\alpha$ denotes the path
connecting $v$ to $v_{0}$ and $\beta$ denotes the path connecting $w$ to $%
v_{0}$. Since $c_{1}$ is an extreme cycle, there is a path $\gamma$ with $%
s(\gamma)=v_{0}$ and $r(\gamma)\in c_{1}^{0}$. Then the path $\beta\gamma$
can be elongated to a path connecting $w$ to $v$. Hence $w\geq v$ which
implies $u\geq v$. It is then clear that $(E\backslash(H,S))^{0}$ satisfies
the strong CSP with respect to $\{v\}$. Since Condition (K) always implies
Condition (L), in both cases $E\backslash(H,S)$ is downward directed
satisfying the strong CSP and Condition (L). Hence, by Theorem \ref{Strongly
primes of LPAs}, $P$ is strongly prime.

(2) Suppose $E$ is an arbitrary graph. If every prime ideal of $L_{K}(E)$ is
strongly prime, then as noted earlier, $E$ satisfies Condition (K) and so
every ideal is graded. So any prime ideal $P$ will be of the form $P=I(H,S)$
where $E\backslash(H,S)$ is downward directed and satisfies Condition (K)
and therefore Condition (L). Thus, in the context of Condition (K), Theorem %
\ref{Strongly primes of LPAs} implies that the condition that the downward
directed graph $E\backslash(H,S)$ satisfies strong CSP is equivalent to the
prime ideal $I(H,S)$ being strongly prime.
\end{proof}

\begin{example}
\textrm{Let $E$ be the following graph. }

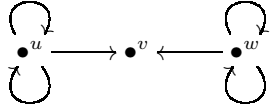
\textrm{\bigskip }

\textrm{%
\begin{equation*}
\xymatrix{ \bullet^{u} \ar@(ul,ur) \ar@(dr,dl) \ar[r] & \bullet^{v} &
\bullet^{w} \ar[l] \ar@(ul,ur) \ar@(dr,dl)}
\end{equation*}
}

\textrm{\bigskip }

\textrm{\noindent The graph is finite and satisfies Condition (K). Hence
every ideal of $L_{K}(E)$ is graded. Now the proper hereditary saturated
subsets of $E^{0}$, are $H_{0}=\emptyset,H_{1}=\{v\},H_{2}=\{u,v\}$ and $%
H_{3}=\{v,w\}$. Hence $I_{i}=<H_{i}>$, $i=0,1,2,3$ are all the proper ideals
of $L_{K}(E)$. It is easy to check that the prime ideals of $L_{K}(E)$ are $%
\{0\}=I_{0}, I_{2}$ and $I_{3}$ as $E\backslash H_{i}$ is downward directed
for $i=0,2,3$. It is also easy to see that each of these $E\backslash H_{i}$
satisfies strong CSP. Hence the prime ideals $I_{i}$, $i=0,2,3$ are actually
strongly prime. So every prime ideal of $L_{K}(E)$ is strongly prime, thus
illustrating Theorem \ref{psp}. }
\end{example}

\section{Insulated Prime Leavitt path algebras and Insulated Prime Ideals}

\noindent For non-commutative rings, the concept of a left/right strongly
prime ring was introduced in \cite{HL} while dealing with Kaplansky's
conjecture on prime von Neumann regular rings. Following this, the
definition of a left/right strongly prime ideal was given in \cite{KW} which
is different from the one introduced in \cite{JOT}. To avoid confusion with
the concept of strongly prime rings and ideals that we investigated in
Section 4, we rename this concept in Definition \ref{HL Definition} below.

\begin{definition}
$($\cite{HL}$)$ Let $R$ be a ring. A right insulator of an element $a\in R$
is defined to be a finite subset $S(a)$ of $R$, such that the right
annihilator $ann_{R}\{ac:c\in S(a)\}=0$.
\end{definition}

Similarly, left insulator of an element can be defined.

\begin{definition}
$($\cite{HL}$)$\label{HL Definition} A ring $R$ is called a right insulated
prime ring if every non-zero element of $R$ has a right insulator.
Equivalently, every non-zero two-sided ideal of such a ring $R$ contains a
finite non-empty subset $S$ whose right annihilator is zero. This finite set 
$S$ is called an insulator of $I$.
\end{definition}

A\textit{\ left insulated prime ring} is defined similarly.

\bigskip

\noindent Following \cite{HL}, Kau\u{c}ikas and Wisbauer \cite{KW} define
the concept of a right/left strongly prime ideals. Again to avoid confusion
with the concept strongly prime ideals of Section 4, we rename these ideals
as indicated below.

\begin{definition}
An ideal $I$ of a ring $R$ is called a right/left insulated prime ideal, if $%
R/I$ is a right/left insulated prime ring.
\end{definition}

In this section, we first describe when a Leavitt path algebra $L_{K}(E)$ is
a left/right insulated prime ring. Interestingly, the distinction between
the left and right insulated primeness vanishes for Leavitt path algebras.
So we may just state $L_{K}(E)$ as being an insulated prime ring. We show
(Theorem \ref{Type I LPA}) that a Leavitt path algebra $L_{K}(E)$ is an
insulated prime ring exactly when $L_{K}(E)$ is a simple ring or $L_{K}(E)$
is isomorphic to the matrix ring $M_{n}(K[x,x^{-1}])$ some integer $n\geq1$.
Equivalent graphical conditions on $E$ are also given. Next, we characterize
the insulated prime ideals $P$ of $L_{K}(E)$. Non-graded insulated prime
ideals of $L_{K}(E)$ are precisely the (non-graded) maximal ideals of $%
L_{K}(E)$. A graded ideal $P$ with $P\cap E^{0}=H$ will be an insulated
prime ideal of $L_{K}(E)$ if and only if $P=I(H,B_{H})$ and $P$ is either a
maximal graded ideal of $L_{K}(E)$ or is a maximal ideal of $L_{K}(E)$
(which is graded). It is then clear that an insulated prime ideal of $%
L_{K}(E)$ is always a prime ideal. Graphically, if $gr(P)=I(H,B_{H})$, then $%
E\backslash(H,B_{H})$ contains only finitely many vertices, is downward
directed and has no non-empty proper hereditary saturated subset of
vertices. Examples are constructed, showing that the concepts of strongly
prime ideals and insulated prime ideals are independent in the case of
Leavitt path algebras.

We first prove the following useful proposition which states that the matrix
rings $M_{\Lambda}(K[x,x^{-1}])$ are precisely the Leavitt path algebras
which are graded-simple but not simple.

\begin{proposition}
\label{graded-simple} The following properties are equivalent for a Leavitt
path algebra $L=L_{K}(E)$:

\begin{enumerate}
\item $L_{K}(E)$ is graded-simple but not simple, that is, $L_{K}(E)$
contains non-zero proper ideals but does not contain any non-zero proper
graded ideals;

\item $E$ is row-finite, downward directed and contains a cycle $c$ without
exits based at a vertex $v$;

\item $L_{K}(E)\cong_{gr}M_{\Lambda}(K[x,x^{-1}])$ under the matrix grading
of $M_{\Lambda}(K[x,x^{-1}])$, where $\Lambda$ is some index set.
\end{enumerate}
\end{proposition}

\begin{proof}
Assume (1). Let $I$ be a proper non-zero ideal of $L=L_{K}(E)$. Then there
is a vertex $u\notin I$. If $P$ is an ideal maximal with respect to $u\notin
P$, then $P$ is a prime ideal (because, if $a\notin P,b\notin P$, then $u\in
LaL+P$ and $u\in LbL+P$ and then $u=u^{2}\in(LaL+P)(LbL+P)=LaLbL+P$. Since $%
u\notin P$, this implies that $aLb\nsubseteqq P$). Since $L_{K}(E)$ is
graded-simple, $P$ must be non-graded and is of the form $%
P=I(H,B_{H})+<p(c)> $ where $c$ is a cycle without exits in $%
E\backslash(H,B_{H})$ based at a vertex $v$, $E^{0}\backslash
H=E\backslash(H,B_{H})^{0}$ is downward directed and $p(x)$ is an
irreducible polynomial in $K[x,x^{-1}]$ (Theorem 2.2). By hypothesis, the
graded ideal $I(H,B_{H})=\{0\}$ and\ so $H=\emptyset $. Consequently, $E^{0}$
is downward directed and $c$ is the only cycle without exits in $E$. We
claim that $E$ is row-finite. Suppose, on the contrary, there is a vertex $w$
which is an infinite emitter. Since $w$ is not in the hereditary set $c^{0}$%
, it follows from the definition of its saturated closure, that $w$ is not
in the saturated closure of $c^{0}$. Hence $w\notin<c^{0}>$. This is a
contradiction, since the non-zero graded ideal $<c^{0}>=L_{K}(E)$, by
hypothesis. We thus conclude that $E$ must be row-finite. This proves (2).

Assume (2). Now, by downward directness, every path in $E$ ends at the
vertex $v$. Then, by Theorem 4.2.12 in \cite{AAS}, $L_{K}(E)=<c^{0}>\cong
M_{\Lambda}(K[x,x^{-1}])$ where $\Lambda$ \ denotes the set of all paths in $%
E$ that end at $v$, but do not go through the entire cycle $c$. It is shown
in (the paragraph \textquotedblleft grading of matrix rings" in Section 2, 
\cite{HR}) that this isomorphism is a graded isomorphism under the
matrix-grading of $M_{\Lambda}(K[x,x^{-1}])$. Hence (3) follows.

Now (3) $\implies$ (1) follows from the fact that $M_{\Lambda}(K[x,x^{-1}])$
is a graded direct sum of copies of $K[x,x^{-1}]$ and that $K[x,x^{-1}]$ has
no non-zero proper graded ideals under its natural $\mathbb{Z}$-grading.
\end{proof}

The next corollary points out a property of non-graded maximal ideals in a
Leavitt path algebra that will be used later.

\begin{corollary}
\label{I maximal => gr(I) maximal graded} If $I$ is a non-graded maximal
ideal of $L=L_{K}(E)$, then $gr(I)$ is a maximal graded ideal of $L$ and $%
L/gr(I)\cong_{gr}M_{\Lambda}(K[x,x^{-1}])$ for some index set $\Lambda$.
\end{corollary}

\begin{proof}
Now $I$ is, in particular, a (non-graded) prime ideal and so $%
I=I(H,B_{H})+<p(c)>$, where $(E\backslash(H,B_{H}))^{0}=E^{0}\backslash H$
is downward directed, $c$ is a cycle without exits in $E\backslash(H,B_{H})$
based at a vertex $v$ and $p(x)$ is an irreducible polynomial in $%
K[x,x^{-1}] $. If $A=I(H^{\prime}, S^{\prime})$ is a graded ideal such that $%
A\supsetneqq I(H,B_{H})$ then $H^{\prime}\supsetneqq H$. If $u\in
H^{\prime}\backslash H$, then by downward directness, $u\geq v$ and this
implies that ($v$ and hence) $c\in A$ and so $I\subsetneqq A$. By the
maximality of $I$, $A=L$. Thus $gr(I)=I(H,B_{H})$ is a maximal graded ideal
of $L$. Then, by Proposition \ref{graded-simple}, $L/gr(I)\cong_{gr}M_{%
\Lambda}(K[x,x^{-1}])$ for some index set $\Lambda$.
\end{proof}

\begin{theorem}
\label{Type I LPA} The following properties are equivalent for a Leavitt
path algebra $L=L_{K}(E)$ of an arbitrary graph $E$:

\begin{enumerate}
\item $L$ is a left/right insulated prime ring;

\item Either (a) $L$ is a simple ring or (b) $L\cong_{gr}M_{n}(K[x,x^{-1}])$
where $n$ is some positive integer;

\item Either (a) $E$ satisfies Condition (L) and has no non-empty proper
hereditary saturated subsets of vertices or (b) $E$ is a finite ``comet",
that is, a downward directed finite graph containing a cycle without exits.
\end{enumerate}
\end{theorem}

\begin{proof}
Assume (1) and that $L$ is a right insulator ring. If $L$ is a simple ring,
we have nothing to prove. Assume $L$ is not a simple ring. We claim that $L$
is graded-simple. Suppose, on the contrary, there is a non-zero graded ideal 
$I$ in $L$. By hypothesis, $I$ has a right insulator $S$. Now, by Theorem
2.5.22 in \cite{AAS}, the graded ideal $I$ is isomorphic to a Leavitt path
algebra and hence has local units. This means that corresponding to the
finite subset $S$, there is an idempotent $\varepsilon$ (depending on $S$)
in $I$ such that $S\subseteq\varepsilon I\varepsilon$. Let $\varepsilon={%
\displaystyle\sum \limits_{j=1}^{n}}\alpha_{j}\beta_{j}^{\ast}\in I$ \ and
let $X=\{s(\alpha _{j}),r(\alpha_{j})=r(\beta_{j}),s(\beta_{j}):j=1,...,n\}$%
. Now $E^{0}$ cannot be an infinite set, because, otherwise, we can find a
vertex $v\in E^{0}\backslash X$ and as $\varepsilon v=0$, $L\varepsilon\cdot
vL=0$ which implies that $S\cdot vL=0$, a contradiction\ to the hypothesis.
Thus $E^{0}$ is finite. This means that $L$ has a multiplicative identity $1=%
{\displaystyle\sum\limits_{v\in E^{0}}}v$. Moreover, $S\subseteq\varepsilon
I\varepsilon$ implies that $(1-\varepsilon)S=0=S(1-\varepsilon)$. Since $S$
is the insulator for $I$, $1-\varepsilon=0$ or $1=\varepsilon\in I$. Hence $%
I=L$, thus proving that $L$ is graded-simple. We then appeal to Proposition %
\ref{graded-simple} to conclude that $L\cong_{gr}M_{\Lambda}(K[x,x^{-1}])$.
Since $L$ has a multiplicative identity $1$, the index set $\Lambda$ must be
finite and we conclude that $L\cong_{gr}M_{n}(K[x,x^{-1}])$\ for some
positive integer $n$. This proves (2).

Assume (2). If $L$ is a simple ring, it is trivially insulated prime.
Suppose $L\cong_{gr}M_{n}(K[x,x^{-1}])$ for some positive integer $n$. Now $%
K[x,x^{-1}]$, being an integral domain, is clearly an insulated prime ring.
We show (following the ideas in the proof of Proposition II.1, \cite{HL}),
that $M_{n}(K[x,x^{-1}])$ is also insulated prime. Suppose $0\neq A\in
M_{n}(K[x,x^{-1}])$ \ with a non-zero entry, say $a_{ij}\neq0$. For each $%
i=1,...,n$ and $j=1,...,n$, let $E_{ij}\in M_{n}(K[x,x^{-1}])$ be the matrix
unit having 1 at the $(i,j)$-entry and $0$ everywhere else. Let $0\neq b\in
K[x,x^{-1}]$. Then $\{E_{ij}b:i=1,...,n;j=1,...,n\}$ is an insulator for $A$%
. Because, if $N=(n_{ij})$ is any non-zero matrix with a non-zero entry $%
n_{kl}$, then $AE_{jk}bN\neq0$ since its $(i,l)$-entry is $%
a_{ij}bn_{kl}\neq0 $. This proves (1).

Now (2) $\Longleftrightarrow$ (3) by Theorem 2.9.1 and Theorem 4.2.12 in 
\cite{AAS}.
\end{proof}

In the next theorem we describe the insulated prime ideals of a Leavitt path
algebra.

\begin{theorem}
\label{Type I} The following properties are equivalent for an ideal $P$ of a
Leavitt path algebra $L=L_{K}(E)$ with $P\cap E^{0}=H$:

\begin{enumerate}
\item $P$ is an insulated prime ideal of $L_{K}(E)$;

\item Either $P$ is a maximal ideal or $P$ is not a maximal ideal but a
maximal graded ideal such that $L/P\cong M_{n}(K[x,x^{-1}])$ where $n$ is a
positive integer and, in particular, $E^{0}\setminus P$ is a finite set.
\end{enumerate}
\end{theorem}

\begin{proof}
Assume (1). Suppose $P$ is an insulated prime ideal so that $L_{K}(E)/P$ is
an insulated prime ring. So, by Theorem \ref{Type I LPA}, $L_{K}(E)/P$ is a
simple ring or $L_{K}(E)/P\cong$ $M_{n}(K[x,x^{-1}])$.

If $L_{K}(E)/P$ is a simple ring, then clearly, $P$ is a maximal ideal of $%
L_{K}(E)$.

Suppose $L_{K}(E)/P\cong$ $M_{n}(K[x,x^{-1}])$ for some $n\geq1$. As $%
M_{n}(K[x,x^{-1}])$ is a prime ring, $P$ is clearly a prime ideal. We claim
that $P$ must be a graded ideal. Assume to the contrary that $P$ is a
non-graded ideal. By Theorem \ref{Prime ideals}, we can then write $%
P=I(H,B_{H})+<p(c)>$, where $H=P\cap E^{0}$, $E^{0}\backslash H=(E\backslash
(H,B_{H}))^{0}$ is downward directed, $c$ is a cycle without exits in $%
E\backslash(H,B_{H})$ based at a vertex $v$ and $p(x)\in K[x,x^{-1}]$ is an
irreducible polynomial. Then, in $\overline{L_{K}(E)}=L_{K}(E)/I(H,B_{H})$, $%
\overline{P}=P/I(H,B_{H})=<p(c)>\varsubsetneq\bar{M}=<\{c^{0}\}>$. Now,
being a graded ideal, $(\overline{M})^{2}=\overline{M}$ and this implies
that $(\overline{M}/\overline{P})^{2}=\overline{M}/\overline{P}$ in $%
\overline {L_{K}(E)}/\overline{P}\cong M_{n}(K[x,x^{-1}])$. Since $I^{2}\neq
I $ for any non-zero proper ideal $I$ in $K[x,x^{-1}]$, $M_{n}(K[x,x^{-1}])$
satisfies the same property as the ideal lattices of $M_{n}(K[x,x^{-1}])$
and $K[x,x^{-1}]$ are isomorphic (Proposition \ref{M=<c^0>}(ii)). Hence $%
\overline{M}=\overline{L_{K}(E)}$. By Proposition \ref{M=<c^0>}(ii), $\bar{P}%
=<p(c)>$ is a maximal ideal of $\overline{M}=$ $\overline{L_{K}(E)}$ and so $%
\overline {L_{K}(E)}/\overline{P}$ must be a simple ring. This is a
contradiction, since the ring $M_{n}(K[x,x^{-1}])$ is not simple. Thus $P$
must be a graded ideal of $L_{K}(E)$, say $P=I(H,S)$ where $H=P\cap E^{0}$.
Now, by Proposition \ref{graded-simple}, $L_{K}(E)/P\cong M_{n}(K[x,x^{-1}])$
has no non-zero proper graded ideals and hence $P$ is a maximal graded ideal
of $L_{K}(E)$ and hence $S=B_{H}$ and $P=I(H,B_{H})$. Since $%
M_{n}(K[x,x^{-1}])$ contains a multiplicative identity, $E^{0}\backslash
P=(E\backslash(H,B_{H}))^{0}$ is a finite set. This proves (2)

Assume (2). If $P$ is a maximal ideal, then clearly $P$ is an insulated
prime ideal as the simple ring $L_{K}(E)/P$ is insolated prime. Suppose $P$
is a maximal graded ideal with $E^{0}\backslash P$, a finite set. It is easy
to check that $P$ is a graded prime ideal and, as $L_{K}(E)$ is $\mathbb{Z}$%
-graded, $P$ is a prime ideal of $L_{K}(E)$ (see Proposition II.1.4, Chapter
II in \cite{NvO}). Hence $P=I(H,S)$ where $H=P\cap E^{0}$ and $(E\backslash
(H,S))^{0}$ is downward directed. By the maximality of $P$, $P=I(H,B_{H})$
and $L_{K}(E)/P\cong L_{K}(E\backslash(H,B_{H}))$ has no non-zero proper
graded ideals. Hence, by Proposition \ref{graded-simple}, $L_{K}(E)/P\cong
M_{\Lambda}(K[x,x^{-1}])$, where $\Lambda$ is some index set. Now, by
hypothesis, $(E\backslash(H,B_{H}))^{0}=$ $E^{0}\backslash P$ is a finite
set and so $L_{K}(E\backslash(H,B_{H}))$ has a multiplicative identity.
Hence $\Lambda$ must be a finite set and we conclude that $L_{K}(E)/P\cong
M_{n}(K[x,x^{-1}])$ for some positive integer $n$. This proves (1).
\end{proof}

\begin{remark}
\textrm{The property of being insulated prime is independent of the property
of being strongly prime. Note that any non-graded maximal ideal of L is
insulated prime by Theorem \ref{Type I}, but it is not strongly prime since,
by Theorem \ref{Strongly primes of LPAs}, a strongly prime ideal of $L$ must
be graded. Likewise, a graded prime ideal of the form $I(H,B_{H}\setminus{u}%
) $ is strongly prime since it is graded strongly primitive (Theorem \ref%
{Strongly primes of LPAs}), but is not a maximal graded ideal as it is
properly contained in the ideal $I(H, B_{H})$ and consequently it cannot be
an insulated prime ideal. }
\end{remark}

Next, we give two examples. The first one is example of an insulated prime
that is not strongly prime whereas the second one is example of a strongly
prime ideal that is not insulated prime.

\begin{example}
\textrm{Let $E$ be the following graph. }

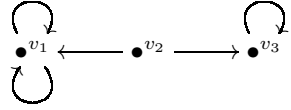
\textrm{\bigskip }

\textrm{%
\begin{equation*}
\xymatrix{ \bullet^{v_1} \ar@(ul,ur) \ar@(dr,dl) & \bullet^{v_2} \ar[r]
\ar[l] & \bullet^{v_3} \ar@(ul,ur) }
\end{equation*}
}

\textrm{\bigskip }

\textrm{\noindent Let us denote the loops on vertex $v_{1}$ by $c_{1}$, $%
c_{2}$ and the loop on vertex $v_{3}$ by $c_{3}$. Now $H=\{v_{1}\}$ is a
hereditary saturated set. Let $P=I(H)+<p(c_{3})>$ where $p(x)$ is an
irreducible polynomial in $K[x,x^{-1}]$. Clearly $P$ is a non-graded ideal.
Now the graph $E\backslash H$ consists of a single edge ending at the cycle $%
c_{3}$ without exits and so $L_{K}(E\backslash H)\cong M_{2}(K[x,x^{-1}])$.
Since $p(x)$ is irreducible in $K[x,x^{-1}]$, $<p(c_{3})>$ will be a maximal
ideal in $L_{K}(E\backslash H)\cong M_{2}(K[x,x^{-1}])$, by Proposition
3.3(ii). Thus $P/I(H)=<p(c_{3})>$ will be a maximal ideal of $%
L_{K}(E)/I(H)\cong L_{K}(E\backslash H)$ and hence $P$ is a maximal
(non-graded) ideal in $L_{K}(E)$. By Theorem \ref{Type I LPA}, $P$ is an
insulated prime ideal. But $P$ is not strongly prime since $P$ is not a
graded ideal. }
\end{example}

\begin{example}
\textrm{Let $E^{0}=\{v_{1},v_{2},v_{3},v_{4}\}$, and the graph $E$ be as
below, 
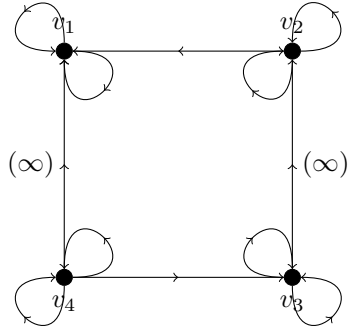
\begin{equation*}
\begin{tikzpicture}[x=3cm, y=3cm,every edge/.style={draw,
postaction={decorate,decoration={markings,mark=at position 0.5 with
{\arrow{>}}}}}] \node[vertex] (v1) at (0,1) [label=above:$v_{1}$] {};
\node[vertex] (v2) at (1,1) [label=above:$v_{2}$] {}; \node[vertex] (v3) at
(1,0) [label=below:$v_{3}$] {}; \node[vertex] (v4) at (0,0)
[label=below:$v_{4}$] {}; \path [->] (v2) edge (v1) (v3) edge [right ] node
{\tt $(\infty)$} (v2) (v4) edge (v3) (v4) edge [left ] node {\tt $(\infty)$}
(v1) (v1) edge [in=180,out=90 ,looseness=20,->] (v1) (v2) edge
[in=90,out=0,looseness=20,->] (v2) (v3) edge [in=0,out=270,looseness=20,->]
(v3) (v4) edge [in=90,out=0,looseness=20,->] (v4) (v1) edge
[in=270,out=360,looseness=20,->] (v1) (v2) edge
[in=180,out=270,looseness=20,->] (v2) (v3) edge [in=90
,out=180,looseness=20,->](v3) (v4) edge [in=180,out=270,looseness=20,->](v4)
; \end{tikzpicture}
\end{equation*}
}

\textrm{\noindent Let $H=\{v_{1},v_{2}\}$. Then $H$ is a hereditary
saturated subset of $E^{0}$ and $B_{H}=\{v_{3},v_{4}\}$. Consider the graded
ideal $P=I(H,B_{H}\backslash\{v_{3}\})$. Now $(E\backslash(H,B_{H}\backslash%
\{v_{3}\}))^{0}=\{v_{3},v_{4},v_{3}^{\prime}\}$ is downward directed and
satisfies (Condition (K) and hence) Condition (L) and the strong CSP with
respect to $\{v_{3}^{\prime}\}$. Hence $P$ is strongly prime by Theorem
4.10. But $P$ is not a maximal graded ideal, as $P\subsetneqq I(H,B_{H})$
and hence $P$ not an insulated prime, by Theorem \ref{Type I LPA}. }
\end{example}

Next, we describe conditions under which an ideal of a Leavitt path algebra
can be factored as a product of (finitely many) insulated prime ideals. In
the proof, we shall be using the following basic result.

\begin{lemma}
\label{star} $($Lemma 2.7.1 \cite{AAS}$)$ A graph $E$ is a comet (with a no
exit cycle $c$ based at a vertex $v$) if and only if $L_{K}(E)\cong
M_{\Lambda}(K[x,x^{-1}])$ where $\Lambda$ is the set of all paths that end
at $v$ but not include the entire cycle $c$.
\end{lemma}

\begin{theorem}
\label{I = Product of Insul. Primes} The following properties are equivalent
for an ideal $I$ of a Leavitt path algebra $L=L_{K}(E)$ with $gr(I)=I(H,S)$:

\begin{enumerate}
\item $I$ is a product of (finitely many) insulated prime ideals of $L$;

\item $I(H,S)=gr(I)=Q_{1}\cap\cdots\cap Q_{m}$ is an irredundant
intersection of $m$ graded ideals each of which is either an insulated prime
(hence maximal graded) ideal or a maximal graded ideal which is not
insulated prime;

\item $I(H,S)=gr(I)=Q_{1}\cap\cdots\cap Q_{m}$ is an irredundant
intersection of $m$ graded ideals each of which is either an insulated prime
(which is maximal graded) ideal or a maximal graded ideal which is not
insulated prime and $I=I(H,S)+\Sigma_{t=1}^{k}<f_{t}(c_{t})>$, where $k\leq
m $, for each $t=1,\ldots, k$, $c_{t}$ is a cycle without exits in $%
E\backslash(H,S)$ based at a vertex $v_{t}$ and $f_{t}(x)\in K[x]$ with a
non-zero constant term.

\item $L/I(H,S)=L_{1}\oplus\cdots\oplus L_{m}$, where $\oplus$ is a graded
ring direct sum and, for each $j=1,\ldots, m$, $L_{j}\cong
M_{\Lambda_{j}}(K[x,x^{-1}])$ where $\Lambda_{j}$ is a finite or infinite
index set;

\item The quotient graph $E\backslash(H,S)$ is an irredundant union of
finitely many finite comets.
\end{enumerate}
\end{theorem}

\begin{proof}
Assume (1) so $I=P_{1}\cdots P_{n}$ is a product of insulated prime ideals $%
P_{j}$. By Theorem \ref{Type I}, each ideal $P_{j}$ is either a maximal
ideal or a maximal graded ideal such that $L/P_{j}\cong
M_{n_{j}}(K[x,x^{-1}])$ \ where $n_{j}$ is a positive integer. Let $%
Q_{j}=gr(P_{j})$ $\ $for $j=1,\ldots, n$. Then $g(I)=Q_{1}\cdots Q_{n}=$ $%
Q_{1}\cap\cdots\cap Q_{n}$, by Lemma \ref{Property of graded ideal}. If
necessary, after removing appropriate ideals and after re-indexing, we get $%
g(I)=$ $Q_{1}\cap\cdots\cap Q_{m}$, an irredundant intersection of graded
ideals with $m\leq n$. Here, for each $j=1,\ldots,m$, either $Q_{j}$ is a
graded insulated prime ideal or $Q_{j}=gr(P_{j})$ is a maximal graded ideal
which is not insulated prime with $P_{j}$ a non-graded maximal ideal of $L$
(Corollary \ref{I maximal => gr(I) maximal graded}. This proves (2).

Now (2) $\Longleftrightarrow$ (3). Because, a maximal graded ideal of $L$ is
a prime ideal and so the equivalence of (2) and (3) follows from the
equivalence of conditions (2) and (3) of Theorem \ref{I intersection of
irreducibles}.

Assume (2). Then, in $\bar{L}=L/I(H,S)$, $\bar{0}=\bar{Q}_{1}\cap\cdots\cap%
\bar{Q}_{m}$, where, for each $j=1,\ldots,m$, $\bar {Q}_{j}=$ $Q_{j}/I(H,S)$%
. Consider the map $\theta:\bar{L}\longrightarrow$ $\bar{L}/\bar{Q}%
_{1}\oplus\cdots\oplus$ $\bar{L}/\bar{Q}_{m}$ given by $a\longmapsto(a+\bar{Q%
}_{1},\ldots, a+\bar{Q}_{m})$, where $\oplus$ is a graded direct sum. Now $%
\theta$ is clearly a monomorphism. It is also a graded morphism, since the
coset map $\bar{L}\longrightarrow\bar{L}/\bar {Q}_{j}$ is a graded morphism
for all $j$. To show that $\theta$ is an epimorphism, first note that the
Chinese Remainder Theorem holds in the Leavitt path algebra $\bar{L}\cong
L_{K}(E\backslash(H,S))$ (see Remark after Theorem 4.3 n \cite{R-2}). Also,
, by maximality, $\bar{Q}_{i}+\bar{Q}_{j}=\bar{L}$ for all $i,j$ with $i\neq
j$. Consequently, given any element $x=(x_{1}+\bar{Q}_{1},\ldots, x_{m}+\bar{%
Q}_{m})\in\bar{L}/\bar{Q}_{1}\oplus\cdots\oplus$ $\bar{L}/\bar{Q}_{m}$,
there is an element $y\in\bar{L}$ such that $y\equiv x_{j}(mod\;\bar{Q}_{j})$
for all $j=1,\ldots, m$. It is then clear that $\theta(y)=x$. Thus $\theta$
is a graded isomorphism and $\bar{L}\cong_{gr}\bar{L}/\bar{Q}%
_{1}\oplus\cdots\oplus$ $\bar{L}/\bar{Q}_{m}$. If the graded ideal $Q_{j}$
is an insulated prime ideal, then, by Theorem \ref{Type I} $\bar{L}/\bar{Q}%
_{j}\cong L/Q_{j}\cong M_{n_{j}}(K[x,x^{-1}])$ where $n_{j}$ is some
positive integer. On the other hand, if $Q_{j}=gr(P_{j})$ is a maximal
graded ideal which is not insulated prime (with $P_{j}$ a non-graded maximal
ideal of $L$), then, by Corollary \ref{I maximal => gr(I) maximal graded}, $%
\bar{L}/\bar{Q}_{j}\cong L/Q_{j}\cong M_{\Lambda_{j}}(K[x,x^{-1}])$ where $%
\Lambda_{j}$ is an infinite index set. This proves (4).

Assume (4) so $\bar{L}=L/I(H,S)=L_{1}\oplus\cdots\oplus L_{m}$ where $\oplus$
is a graded ring direct sum and, for each $j=1,\ldots, m$, $L_{j}\cong
M_{\Lambda_{j}}(K[x,x^{-1}])$ where $\Lambda_{j}$ is a finite or infinite
index set. As $L_{j}\cong M_{\Lambda_{j}}(K[x,x^{-1}])$ is graded-simple, we
have for each $j=1,\ldots,m$, $A_{j}=\oplus_{i\neq j,i=1}^{i=m}L_{i}$ $\ $is
a maximal graded ideal of $\bar{L}\cong L_{K}(E\backslash(H,S))$ and so will
be of the form $A_{j}=I(H_{j},B_{H_{j}})$ where $H_{j}=A_{j}\cap(E%
\backslash(H,S))^{0}$. Now $L_{K}(E\backslash (H,S))/I(H_{j},B_{H_{j}})\cong%
\bar{L}/A_{j}\cong L_{j}\cong M_{\Lambda_{j}}(K[x,x^{-1}])$ and so, by Lemma %
\ref{star}, $M_{j}=(E\backslash (H,S))\backslash(H_{j},B_{H_{j}})$ is a
comet. Since $\cap_{j=1}^{m}A_{j}=\{0\}$, $E\backslash(H,S)^{0}=%
\cup_{j=1}^{m}M_{j}$, a union of finite number of comets. This proves (5).

Assume (5). So $E\backslash(H,S)=\cup_{j=1}^{m}M_{j}$ is a union of comets.
If $H_{j}=(E\backslash(H,S))\backslash M_{j}$, then $\bar{Q}%
_{j}=I(H_{j},B_{H_{j}})$ is a graded ideal of $\bar{L}=L_{K}(E%
\backslash(H,S))\cong L/I(H,S)$ and $\cap_{j=1}^{m}\bar{Q}_{j}=0$. Then 
\begin{equation*}
\bar{L}/\bar{Q}_{j}\cong
L_{K}[(E\backslash(H,S))\backslash(H_{j}.B_{H_{j}})]\cong L_{K}(M_{j})\cong
M_{\Lambda_{j}}(K[x,x^{-1}])
\end{equation*}
where $\Lambda_{j}$ is an finite or infinite index set according as $M_{j}$
is a finite or infinite comet. Now, for each $j$, $\bar{Q}_{j}=Q_{j}/I(H,S)$
for some ideal $Q_{j}\supseteq I(H,S)$. Clearly, $gr(I)=I(H,S)=%
\cap_{j=1}^{m}Q_{j}$ where \ the $Q_{j}$ are graded ideals which are either
insulated prime or non-maximal ideals which are maximal graded ideals
according as $\Lambda_{j}$ is a finite or an infinite index set. This proves
(2).
\end{proof}

\begin{remark}
As noted in Remark \ref{Uniques prod. str. primes}, factorization of an
ideal $I$ of $L$ as an irredundant product or intersection of finitely many
insulated prime ideals is unique except for the order of the factors due to
the fact that an insulated prime ideal is always a prime ideal.
\end{remark}

We next consider the case when every ideal of $L$ is a product of finitely
many insulated prime ideals.

\begin{theorem}
\label{Every ideal insulated prime} The following properties are equivalent
for any Leavitt path algebra $L=L_{K}(E)$:

\begin{enumerate}
\item Every ideal of $L$ is a product of (finitely many) insulated prime
ideals;

\item $L\cong_{gr}\oplus_{j=1}^{m}L_{j}$ is a graded ring direct sum of
matrix rings $L_{j}\cong_{gr}M_{\Lambda_{j}}(K[x,x^{-1}])$ where $%
\Lambda_{j} $ is a suitable index set.
\end{enumerate}
\end{theorem}

\begin{proof}
Assume (1). Since $\{0\}$ is a product of insulated prime ideals, applying
Theorem \ref{I = Product of Insul. Primes} (iv) \ with $I=\{0\}$, we obtain
(2).

Assume (2), so that $L=L_{1}\oplus\cdot\cdot\cdot\oplus L_{m}$, where $%
\oplus $ is a graded ring direct sum and, for each $j=1,\cdot\cdot\cdot,m$, $%
L_{j}\cong M_{\Lambda_{j}}(K[x,x^{-1}])$ where $\Lambda_{j}$ is a finite or
infinite index set. Note that $\{0\}$ is a maximal graded ideal of $%
K[x,x^{-1}]$ and so is an insulated prime ideal. Also every non-zero ideal
of $K[x,x^{-1}]$ is a product of maximal (hence insulated prime) ideals of $%
K[x,x^{-1}]$. By Proposition \ref{M=<c^0>}(ii), the same properties hold for
ideals of $L_{j}\cong M_{\Lambda_{j}}(K[x,x^{-1}])$ for all $j=1,\ldots, m$.
Since $\oplus$ is a ring direct sum, a simple induction on $m$ shows that
every ideal of $L$ is a product of insulated prime ideals. This proves (1).
\end{proof}

As an illustration, we have the following example.

\begin{example}
\textrm{Let }E = F $\cup $ F\textrm{\ be the disjoint union of two copies of
\ the graph F as shown below:}

\bigskip

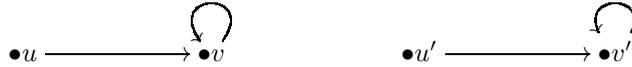
\[
\xymatrix{ && \bullet{u}\ar@{->}[rr] && \bullet{v} \ar@(ur,ul)}\xymatrix{ && \bullet{u'}\ar@{->}[rr] && \bullet{v'}\ar@(ur,ul)}
\]

\bigskip

\noindent Now $L_{K}(F)\cong _{gr}M_{2}(K[x,x^{-1}]$ and so $L_{K}(E)$\textrm{\bigskip 
}$\cong _{gr}(M_{2}(K[x,x^{-1}]\oplus M_{2}(K[x,x^{-1}])$.
\end{example}

The next theorem describes when every ideal of a Leavitt path algebra is
insulated prime.

\begin{theorem}
Let $L=L_{K}(E)$ be any Leavitt path algebra. Then we have the following:

\begin{enumerate}
\item Every ideal of $L$ is insulated prime if and only if $L$ is a simple
ring;

\item Every non-zero ideal of $L$ is insulated prime if and only if either $%
L $ is a simple ring \ or $L$ contains exactly one non-zero ideal $I$ which
is graded and $L/I\cong M_{n}(K[x,x^{-1}])$ for some positive integer $n$.
\end{enumerate}
\end{theorem}

\begin{proof}
(1) Assume every ideal of $L$ is insulated prime. We claim $L$ is a simple
ring. Suppose, by way of contradiction, $L$ contains non-zero ideals. Since $%
\{0\}$ is an insulated prime and not a maximal ideal of $L$, it follows from
Theorem \ref{Type I}, that $L\cong M_{n}(K[x,x^{-1}])$ for some positive
integer $n$. But $M_{n}(K[x,x^{-1}])$ contains non-zero ideals which are not
maximal and if they were to be insulated prime, they must be (non-zero)
maximal graded ideals, by Theorem \ref{Type I}. But this is not possible,
since $M_{n}(K[x,x^{-1}])$ is graded-simple, that is, it contains no
non-zero graded ideals. Hence $L$ must be a simple ring. The converse is
obvious.

(2) Suppose every non-zero ideal is insulated prime. If $L$ is a simple
ring, we are done. Suppose $L$ contains non-zero ideals. We first claim that
every ideal of $L$ must be graded. Suppose, on the contrary, there is a
non-graded ideal $I$. By Theorem \ref{genators of ideal}, $I$ will be of the
form $I=I(H,S)+{\displaystyle\sum\limits_{t\in T}}<f_{t}(c_{t})>$ where $%
H=I\cap E^{0}$, $S=\{v\in B_{H}:v^{H}\in I\}$, $T$ is a non-empty index set,
for each $t\in T$, $c_{t}$ is a cycle without exits in $E\backslash(H,S)$
and $f_{t}(x)\in K[x]$ with a non-zero constant term. But then, for a fixed $%
t\in T$, we can construct the ideal $A=I(H,S)+<1-c_{t}^{2}>$ and, as $%
1-x^{2} $ is not an irreducible polynomial in $K[x,x^{-1}]$, $A$ is not a
prime ideal of $L$ and hence not insulated prime. This contradiction shows
that every non-zero ideal $I$ of $L$ is a graded ideal. By Theorem \ref{Type
I}, every non-zero ideal of $L$ must then be a maximal graded ideal such
that $L/I=M_{n}(K[x,x^{-1}])$ \ for some $n>0$. We claim that $L$ has
exactly one non-zero ideal. Suppose, on the contrary, there are two distinct
non-zero ideals $A,B$ in $L$. By the maximality of $A$ and $B$, $A\nsubseteq
B$, $B\nsubseteq A$ and $A+B=L$ . If $A\cap B$ is non-zero, then, by
hypothesis, it is insulated prime and hence prime and so $AB\subseteq A\cap
B $ will imply $A\subseteq A\cap B$ or $B\subseteq A\cap B$, a
contradiction. Hence $A\cap B=0$. But then $L=A\oplus B$ and this again
leads to a contradiction since $A\cong L/B\cong M_{n}(K[x,x^{-1}])$ will
then contain non-graded ideals contradicting the fact that every ideal of $L$
is graded. Thus $L$ contains exactly one non-zero ideal $I$ which is graded
and $L/I\cong M_{n}(K[x,x^{-1}])$ for some positive integer $n$. Since such
an ideal $I$ is always insulated prime, the converse holds.
\end{proof}

\end{document}